\documentclass[10pt]{article}

\usepackage[T1]{fontenc}
\usepackage[active]{srcltx}
\usepackage{lmodern}
\usepackage{amsmath,amstext,amssymb,graphicx,graphics,color}
\usepackage{mathtools}
\usepackage{tikz}
\usepackage{mleftright}
\usepackage{pgfplots}
\usepackage{cite}               
\usepackage{hyperref}           
\usepackage{bbm}                
\usepackage{fancybox}           
\usepackage[section]{placeins}  
\usepackage{subcaption}
\usepackage{enumitem}
\usepackage{amsthm}
\usepackage[margin=1in]{geometry}
\theoremstyle{plain}
\newtheorem{theorem}{Theorem}[section]
\newtheorem{lemma}[theorem]{Lemma}
\newtheorem{proposition}[theorem]{Proposition}

\newcommand \commentout[1] {}

\DeclarePairedDelimiter\abs{\lvert}{\rvert}

\theoremstyle{definition}
\newtheorem{definition}[theorem]{Definition}

\theoremstyle{remark}
\newtheorem{remark}[theorem]{Remark}

\renewcommand{\qed}{\hfill$\square$\smallskip}
\newcommand{\RN}[1]{
  \textup{\uppercase\expandafter{\romannumeral#1}}
}
\newcommand{\expo}{\mathrm{e}}

\newcommand{\dd}{\mathrm{d}}

\newcommand{\R}{{\mathbb{R}}}

\newcommand{\cP}{\mathcal{P}}

\newcommand{\cD}{\mathcal{D}}
\newcommand{\cM}{\mathfrak{M}}
\newcommand{\cH}{\mathcal{H}}
\newcommand{\cQ}{\mathcal{Q}}

\title{The Bennati-Dragulescu-Yakovenko model in the continuous setting: PDE derivation and long-time behavior}
\author{
  Fei Cao \thanks{Department of Mathematics, King Fahd University of Petroleum and Minerals, Dhahran, Eastern Province, Saudi Arabia. Email: \texttt{fei.cao@kfupm.edu.sa}} \and
  Nadia Loy \thanks{Department of Mathematical Sciences, Politecnico di Torino, Corso Duca degli Abruzzi, 24, 10129 Torino, Italy. Email: \texttt{nadia.loy@polito.it}.}
}
\date{\today}

\begin{document}

\maketitle

\begin{abstract}
In this manuscript, we develop and analyze a continuous version of the well-known Bennati–Dragulescu-
Yakovenko (BDY) dollar-exchange discrete model. Starting from the conservative BDY exchange mechanism, we rely on kinetic theory for multi-agent systems in order to propose an analogous Boltzmann-type continuous dynamics, which does not belong to the class of other popular kinetic models for wealth exchange. We employ the quasi-invariant limit procedure to rigorously derive a nonlinear PDE on the half-line, which is a Fokker-Planck equation featuring the boundary value in the drift term. The PDE is supplemented with a nonlinear Robin-type boundary condition encoding conservation of total agents and wealth. We prove relative entropy decay to the unique stationary state that is the Boltzmann-Gibbs (exponential) distribution with a $\expo^{-c\sqrt{t}}$ convergence rate. We determine a new entropy production term which is a Fisher information minus a competing term which measures the temporal distance of the boundary term from its equilibrium value. Thanks to an inequality relating the Fisher information relative to the Boltzmann-Gibbs equilibrium to the $L^1$-norm, we prove $L^1$-convergence toward equilibrium and convergence of the boundary term (up to a subsequence), as well as convergence in both the Hellinger and Fourier distances. This represents an advancement in classical arguments of entropy decay for Fokker-Planck equations in a kinetic framework.

\end{abstract}

\section{Introduction}\label{sec:introduction}

In this paper, we propose and investigate a continuous analog of the famous Bennati-Dragulescu-Yakovenko (BDY) wealth exchange model \cite{dragulescu_statistical_2000}, which serves as a fundamental and pioneering model in the econophysics literature \cite{garibaldi_statistical_2007,scalas_statistical_2006}. In the classical BDY exchange model, there are $N$ distinct agents labeled $1$ through $N$, each described by the number of dollars they hold. Let $X^i_t$ denote the wealth of agent $i$ at time $t$. The model prescribes a simple exchange mechanism within a closed economy \cite{cao_derivation_2021,cao_interacting_2024,lanchier_rigorous_2017}: at random times (governed by an exponential distribution), one agent $i$ is chosen uniformly at random to give a dollar to another randomly selected agent $j$. If the chosen agent $i$ has no money ($X^i=0$), the event is void and no transfer of wealth occurs. The BDY dynamics can be summarized as follows:

\begin{equation}\label{dynamics:BDY}
\textbf{BDY model:} \qquad (X^i,X^j)~ \begin{tikzpicture} \draw [->,decorate,decoration={snake,amplitude=.4mm,segment length=2mm,post length=1mm}]
(0,0) -- (.6,0); \node[above,red] at (0.3,0) {};\end{tikzpicture}~  (X^i-1,X^j+1) \quad (\text{if } X^i \geq 1).
\end{equation}
It is readily seen from the aforementioned set-up that the total wealth of the system is preserved at all times: no money is ever created or destroyed during the evolution of the game. Mathematically, this is expressed by
\begin{equation}\label{eq:preserved_sum}
\frac{1}{N}\left(X^1_t + \cdots + X^N_t\right) = \frac{1}{N}\left(X^1_0 + \cdots + X^N_0\right) \coloneqq \mu \quad \text{for all } t\geq 0,
\end{equation}
where $\mu > 0$ denotes the prescribed (initial) average wealth per agent.

The BDY model described above represents one of the earliest mathematically tractable frameworks in econophysics and has since become a reference point for subsequent rigorous developments \cite{cao_derivation_2021,lanchier_rigorous_2017}. Its defining feature is the unbiased interaction rule: each agent with positive wealth transfers one dollar at a constant rate, with the recipient chosen uniformly at random, so that no individual or subgroup is structurally favored. In this sense, the dynamics is termed the \emph{unbiased exchange model} in \cite{cao_derivation_2021,cao_interacting_2024} or equivalently the \emph{one-coin model} \cite{lanchier_rigorous_2017}. From a statistical physics perspective, interpreting agents as particles and pairwise exchanges between agents as binary collisions between particles, the wealth of an agent naturally corresponds to the velocity carried by a particle. Under this natural analogy, the BDY dynamics connects closely to interacting particle systems and has in fact been investigated within the framework of zero-range processes (with constant rates) \cite{graham_rate_2009,merle_cutoff_2019,morris_spectral_2006}. Natural generalizations of the classical BDY model have also been explored in several directions, including but not limited to
models with bank and debt \cite{cao_bias_2023,cao_uncovering_2022,lanchier_rigorous_2019,lanchier_distribution_2022}, models with probabilistic cheaters \cite{blom_hallmarks_2024,cao_mean_2025}, models with wealth ceiling and wealth flooring \cite{cao_2026_wealth}, poor-biased or rich-biased exchange models \cite{cao_derivation_2021,cao_sticky_2025,cao_uniform_2024,miao_nonequilibrium_2024}.

From a broader perspective, these variants underscore the flexibility of the BDY framework and motivate the search for continuous analogues: by passing from discrete agent-based rules to nonlinear PDEs with suitable (Robin-type) boundary conditions, in this paper we aim to obtain a natural continuum description that retains the conservative and unbiased features of the original BDY exchange dynamics while enabling analytical study of large-time behavior.


It is well known in both the econophysics and zero-range process literature that, in the large population limit $N \to \infty$, the behavior of the wealth of any fixed agent $X_t^i$ under the mean-field BDY dynamics can be summarized as follows \cite{cao_derivation_2021,cao_interacting_2024,graham_rate_2009,merle_cutoff_2019,morris_spectral_2006}: Let ${\bf p}(t)=\big(p_0(t),p_1(t),\ldots,p_n(t),\ldots\big)$ be a time-evolving probability mass function on $\mathbb N$ in which $p_n(t)$ represents the fraction of agents (among a large pool of agents as the number of agents $N \to \infty$) having $n$ dollars at time $t$. Then its time evolution is governed by the following Boltzmann-type infinite system of nonlinear ODEs:
\begin{equation}\label{eq:law_limit_unbais}
p'_n = \left\{
    \begin{array}{ll}
      p_1-r\,p_0 & \quad \text{for } n=0, \\
      p_{n+1}+r\,p_{n-1}- (1+r\,)p_n & \quad \text{for } n \geq 1,
    \end{array}
  \right.
\end{equation}
where $r \coloneqq \sum_{n\geq 1} p_n$ represents the fraction of agents who are ``rich enough'' to give out a dollar. Moreover, assume that ${\bf p}(0) \in \mathcal{P}(\mathbb N)$ is a probability mass function on $\mathbb N$ with mean value $\mu > 0$, then the (classical) solution ${\bf p}(t)$ of the mean-field BDY ODE system \eqref{eq:law_limit_unbais} enjoys several fundamental properties, which we summarize below for the reader's convenience:
\begin{enumerate}[label=(\Alph*)]
\item The system \eqref{eq:law_limit_unbais} preserves the total probability mass and the mean value. In other words, ${\bf p}(t) \in \mathcal{P}(\mathbb N)$ is a probability mass function on $\mathbb N$ with mean value $\mu$ for all $t \geq 0$.
\item The geometric distribution ${\bf p}^*$, defined by $p^*_n = \frac{1}{1+\mu}\,\left(\frac{\mu}{1+\mu}\right)^n$ for $n\geq 0$, is the unique equilibrium solution to which ${\bf p}(t)$ converges (in the sense of relative entropy). Moreover, in the limit as $\mu \to \infty$, the geometric distribution ${\bf p}^*$ can be well-approximated by an exponential distribution with mean $\mu$.
\item The relative entropy from the solution ${\bf p}(t)$ to the geometric equilibrium solution ${\bf p}^*$, defined by \[\cH\left({\bf p}(t)\mid {\bf p}^*\right) \coloneqq \sum_{n\geq 0} p_n(t)\,\ln \frac{p_n(t)}{p^*_n},\] decreases monotonically with respect to time.
\end{enumerate}

The aim of the present work is to describe a continuous version  of the BDY model~\eqref{dynamics:BDY} and correspondingly of the mean-field nonlinear ODE system \eqref{eq:law_limit_unbais}.
To this aim, we shall resort to a kinetic approach. Based on the idea that an economic system composed by a sufficiently large
number of agents can be described using the laws of statistical mechanics
as it happens in a physical system composed of many interacting particles, kinetic theory has proved to be an efficient framework for the description of socio-economic phenomena, which can be modelled as multi-agent systems in which agents interact binarily according to universal rules~\cite{cao_fractal_2024,cao_iterative_2024,pareschi_interacting_2013}. The description of wealth distribution is one of the key applications~\cite{bato, cao_explicit_2021,during_kinetic_2008,CPT, matthes_2008,matthes_steady_2008, torregrossa_wealth_2017,torregrossa_wealth_2018, toscani_2009}, which has also been generalized to international markets with trades and migrations~\cite{bisi_2022,bisi_loy_2024}. These models are collision-like kinetic equations of Boltzmann type, which implement linear exchange dynamics, whereby two agents transfer a \textit{fraction} of their respective wealth to one another. The exchange of a fraction of wealth is the key aspect of the popular exchange model by Chakraborti and Chakrabarti~\cite{chak_chak} (and its variants), which inspired many of the above mentioned kinetic models. Conversely, the BDY exchange prescribes an \textit{absolute exchange} of a fixed amount of money.

In the kinetic framework, one of the crucial issues is the rigorous analysis of the impact of the microscopic interactions between agents on the characterization of their aggregate wealth distribution in the long-time. To this respect, one of the core issues is the derivation of Fokker-Planck type equations from Boltzmann-type collisional kinetic ones, which are integro-partial differential equations, and typically hardly tractable. This derivation can be carried out by means of the \textit{quasi-invariant limit technique}\cite{toscani_2006}, built upon the \textit{grazing collision limit}\cite{desvillettes_1992,desvillettes_2001,villani_1998} in gas dynamics, which relies on considering interactions which produce small changes and to analyze them on a suitable slow time scale which compensates for such smallness, thus allowing enough interactions to take place in order to observe the emerging aggregate trend~\cite{toscani_2006,pareschi_interacting_2013}. Fokker-Planck equations, which are more amenable to analytical investigations, allow more easily to determine the stationary asymptotic statistical profile of the wealth distribution. For example, the Fokker-Planck equation which can be derived from the kinetic description of linear exchange rules allows to show that the stationary state is a Gamma inverse distribution featuring fat power law tails which reproduce the inverse power law of wealth observed by Vilfredo Pareto\cite{during_kinetic_2008,CPT}. Another crucial problem in kinetic theory, is the decay to equilibrium, which is typically studied by analyzing the monotonicity of Lyapunov entropy functionals~\cite{furioli_2017}. This issue has been addressed for Fokker-Planck equations with both constant and non-constant diffusions, and with a linear drift~\cite{furioli_2017,loy_zanella_2021,torregrossa_wealth_2017}, but is still an open problem for general nonlinear Fokker-Planck equations.

As done, for example, in~\cite{CPT, torregrossa_wealth_2017,torregrossa_wealth_2018}, we shall consider a continuous wealth $v \in \R_+$ and a corresponding density of agents $f(v,t)$ with personal wealth $v \ge 0$ at time $t \ge 0$, but the microscopic dynamics does not rely on the classical linear exchange. In section~\ref{sec:sec2}, we derive, by means of the quasi-invariant limit, the following nonlinear PDE for the temporal evolution of $f(\cdot,t)$ subject to a nonlinear Robin-type boundary condition:
\begin{equation}\label{eq:mainPDE}
\begin{cases}
\partial_t f(v,t) = \partial_{vv} f(v,t) + f(0,t)\,\partial_v f(v,t), &~~ v > 0,~t\geq 0,\\
\partial_v f(v,t) + f^2(v,t) = 0, &~~ v = 0,~t\geq 0.
\end{cases}
\end{equation}
Heuristically speaking, the nonlinear PDE~\eqref{eq:mainPDE} can be regarded as the continuous and infinitesimal analogue of the classical mean-field BDY ODE system \eqref{eq:law_limit_unbais}, as we explain in detail in Section~\ref{sec:sec2}. From this point onward, we refer to the nonlinear PDE \eqref{eq:mainPDE} as the \textbf{Bennati-Dragulescu-Yakovenko (BDY) PDE}. The BDY PDE is a nonlinear Fokker-Planck equation, which features a constant diffusion but a time-varying drift coefficient which involves the boundary value $f(0,t)$ of the density itself. This relates to the importance of the action of the drift term on the boundary conditions, which has been shown in~\cite{feller_1951}. 
The nonlinear PDE~\eqref{eq:mainPDE} was postulated in~\cite{calvez_one_2010} as a Keller-Segel like model for cell migration on the half line in response to a chemo-attractant with production term on the boundary. This is done by setting the drift term equal to the trace value. This equation has been further studied in~\cite{calvez_analysis_2012,lepoutre_cell_2014} also for cases in which the mass of the solution is not unitary. 


The remainder of this manuscript is organized as follows. Section~\ref{sec:sec1} introduces a Boltzmann-type model which describes the continuous counterpart of the BDY mechanism and provides basic results on existence, uniqueness and general properties. Section~\ref{sec:sec2} provides a rigorous justification of the BDY PDE~\eqref{eq:mainPDE} by virtue of a quasi-invariant limit procedure commonly encountered in the kinetic theory for multi-agent systems. In parallel to the list of key properties satisfied by the solution ${\bf p}(t)$ of the BDY ODE system \eqref{eq:law_limit_unbais}, we show in subsequent sections that the solutions $f$ of the PDE \eqref{eq:mainPDE} also satisfy the analogue of properties (A)-(C). The large-time behavior of solutions $f(\cdot,t)$ to~\eqref{eq:mainPDE} is examined in sections~\ref{subsec:sec3.1}--\ref{subsec:sec3.2}--\ref{subsec:sec3.3}. There we show that the evolution dissipates the relative entropy with respect to its unique equilibrium $f^\infty$ (an exponential density), we establish the appropriate relation between the Fisher information and the $L^1$, and between the Hellinger distance and the entropy production term. An $\expo^{-c\sqrt{t}}$ convergence rate is established, under the assumption of exponential control of the solution. Section~\ref{subsec:sec3.4} turns to the linearized dynamics and proves a quantitative exponential decay estimate in a suitable weighted $L^2$ space. Finally, section \ref{sec:sec4} concludes the manuscript by outlining several potential directions for future research, building upon the contributions made in this work.


\section{A kinetic description of the BDY exchange model}\label{sec:sec1}

\subsection{A collision-like kinetic model : the $\varepsilon$ BDY-PDE}\label{subsec:sec2.1}
This subsection is dedicated to the formal derivation of the nonlinear PDE \eqref{eq:mainPDE}. First we recall that the dynamics of the agent-based BDY model involve selecting two agents uniformly at random, with one agent transferring one dollar (if possible) to the other. Our derivation is split into two steps. The first step consists in constructing a continuous analogue of the basic BDY microscopic mechanism~\eqref{dynamics:BDY} which allows agents to trade a (potentially tiny) amount $\varepsilon > 0$ of dollars (if feasible) in each binary transaction, and which permits the wealth of agents to vary continuously in the interval $[0,\infty)$.  Due to the nature of this microscopic dynamics, a suitable framework for defining a mesoscopic description which naturally incorporates such binary microscopic interactions is represented by Boltzmann-type kinetic equations for binary ``collisions'' or interactions~\cite{pareschi_interacting_2013}. As a second step, we shall derive a Fokker-Planck equation from the Boltzmann-type equation.

First, we start by considering a collision-like kinetic approach, in which two agents $v$ and $w$ may interact and exchange money. The post-transaction wealth which mimics~\eqref{dynamics:BDY} is given by
\begin{equation}\label{micro.rule}
\begin{cases}
v'&= v-\varepsilon,~~{\rm if }~ v \geq \varepsilon,\\
w'&= w+\varepsilon,
\end{cases}
\end{equation}
while $v'=v$ and $w'=w$ if $v < \varepsilon$, i.e., if the agent $v$ has not enough money, the exchange does not take place. The microscopic rule~\eqref{micro.rule} is affine, but differs from the classical linear exchange~\cite{CPT}, as the first agent gives an absolute amount $\varepsilon$, and not a portion of its own wealth.
The asymmetric rule~\eqref{micro.rule} conserves the average amount of money within an exchange since
\begin{equation}
v'+w' = v+w,
\end{equation}
while it does not conserve the second moment for a generic $\varepsilon>0$:
\begin{equation}
v'^2+w'^2 =v^2+w^2 +2\,\varepsilon\,\left(\varepsilon-(v-w)\right)
\end{equation}
when $v \geq \varepsilon$. Let us now introduce $f_\varepsilon: \R_+\times \R_+ \to \R_+, (v,t) \to f_\varepsilon(v,t)$, which represents the density of the wealth $v$ of a typical agent at time $t$.
The binary microscopic dynamics~\eqref{micro.rule}, which we consider to happen with a frequency $\lambda$, can be described by a Boltzmann-type collisional equation whose weak form reads as
\begin{equation}\label{eq:Boltz.coll}
\dfrac{\dd}{\dd t} \int_{\R_+} f_\varepsilon(v,t)\, \varphi(v) \, \dd v = \left\langle Q_\varepsilon(f_\varepsilon,f_\varepsilon),\varphi \right\rangle,
\end{equation}
in which
\begin{equation}\label{Boltz1}
\begin{split}
 \left\langle Q_\varepsilon(f_\varepsilon,f_\varepsilon),\varphi \right\rangle &= \dfrac{\lambda}{2}\int_{\R_+} \int_{\R_+} B_\varepsilon(v)\, \left[\varphi(v')-\varphi(v)\right]\,  f_\varepsilon(v,t)\,  f_\varepsilon(w,t) \, {\rm d} v \,{\rm d} w\\
&\phantom{=}+\dfrac{\lambda}{2}\int_{\R_+} \int_{\R_+} B_\varepsilon(v)\, \left[\varphi(w')-\varphi(w)\right]\,  f_\varepsilon(v,t)\,  f_\varepsilon(w,t) \, {\rm d} v\, {\rm d} w.
\end{split}
\end{equation}
Here $\varphi$ is a test function of the observable $v$, and the two terms on the right-hand side of \eqref{Boltz1} take into account the asymmetry of the binary interaction~\eqref{micro.rule}. The function $B_\varepsilon$ is the \textit{interaction kernel}, which discriminates whether the exchange takes place or not within a binary interaction, and it is given by
\begin{equation}\label{kernel}
B_\varepsilon(v) \coloneqq \mathbbm{1}_{v\geq \varepsilon}.
\end{equation}
The operator~\eqref{Boltz1} relies on an \textit{assumption} of propagation of chaos \cite{sznitman_topics_1991}, motivated by the propagation of chaos in the discrete model \cite{cao_derivation_2021,cao_interacting_2024,graham_rate_2009,merle_cutoff_2019}.
Inserting \eqref{micro.rule} into \eqref{Boltz1}, we obtain
\begin{equation}\label{Boltz2}
\begin{split}
 \left\langle Q_\varepsilon(f_\varepsilon,f_\varepsilon),\varphi \right\rangle &= \dfrac{\lambda}{2}\int_{\R_+} \int_{\R_+} \left[\varphi(v-\varepsilon)-\varphi(v)\right]\,\mathbbm{1}_{v\geq \varepsilon}\, f_\varepsilon(v,t)\,f_\varepsilon(w,t) \, {\rm d} v
  \,{\rm d} w\\
&\phantom{=}+\dfrac{\lambda}{2}\int_{\R_+} \int_{\R_+} \left[\varphi(w+\varepsilon)-\varphi(w)\right]\,\mathbbm{1}_{v\geq \varepsilon}\, f_\varepsilon(v,t)\,  f_\varepsilon(w,t) \, {\rm d} v\, {\rm d} w.
\end{split}
\end{equation}
A routine change of variables in~\eqref{Boltz2} leads us to
\begin{equation}\label{Boltz3}
\begin{split}
\left\langle Q_\varepsilon(f_\varepsilon,f_\varepsilon),\varphi \right\rangle &= \dfrac{\lambda}{2}\int_{\R_+} \varphi(v)\,\left[f_\varepsilon(v+\varepsilon,t)-f_\varepsilon(v,t)\,\mathbbm{1}_{v\geq \varepsilon}\right] \, {\rm d} v\\
&\phantom{=}+\dfrac{\lambda}{2}\,r[f_\varepsilon](t)\,\int_{\R_+} \varphi(w)\,\left[f_\varepsilon(w-\varepsilon,t)\,\mathbbm{1}_{w\geq \varepsilon}-f_\varepsilon(w,t)\right] \, {\rm d} w,
\end{split}
\end{equation}
where
\begin{equation}\label{eq:proportion_of_rich}
r[f_\varepsilon](t) \coloneqq \int_{\varepsilon}^\infty f_\varepsilon(v,t)\,\dd v  
\end{equation}
represents the proportion of agents who are ``wealthy enough'' to give $\varepsilon$ dollars in a binary exchange event at time $t$. For future purposes, we remark that for $\varepsilon$ small enough, we can Taylor expand $f_\varepsilon$ and obtain the following approximation
\begin{equation}\label{eq:Taylor2}
r[f_\varepsilon](t) = \int_{\varepsilon}^\infty f_\varepsilon(v,t)\,\dd v = 1-\int_0^{\varepsilon} f_\varepsilon(v,t)\,\dd v
= 1 - \varepsilon\,f_\varepsilon(0,t) - \frac{\varepsilon^2}{2}\,\partial_v f_\varepsilon(\bar{v}_\varepsilon,t), 
\end{equation}
where $\bar{v}_\varepsilon = \alpha \varepsilon$ for some $\alpha \in (0,1)$. Therefore, the strong form of equation~\eqref{Boltz3} is given by the following nonlinear PDE:
\begin{equation}\label{eq:epsilon_BDY}
\partial_t f_\varepsilon(v,t) = Q_\varepsilon(f_\varepsilon,f_\varepsilon), \quad Q_\varepsilon(f_\varepsilon,f_\varepsilon) := \dfrac{\lambda}{2}\,\Big[f_\varepsilon(v+\varepsilon,t)- r[f_{\varepsilon}]\,f_\varepsilon(v,t) + \big(r[f_\varepsilon]\,f_\varepsilon(v-\varepsilon,t)\, - f_\varepsilon(v,t)\big)\mathbbm{1}_{v\geq \varepsilon} \Big].
\end{equation}
For convenience, we refer to the PDE \eqref{eq:epsilon_BDY} as the \textbf{$\varepsilon$-BDY PDE}.

\subsection{Main properties and formal limit}

Let us now introduce, for each $s \in \mathbb N_+$, the space
\[\mathcal{P}_s(\R_+) \coloneqq \left\{\nu\in\mathcal{P}(\R_+)\,:\,\int_{\R_+}|v|^s\,\dd \nu(v) < +\infty\right\},\]
where $\mathcal{P}(\R_+)$ denotes the space of probability measures on $\R_+$. We denote the $n$-th (raw) moment of $f$ by
\begin{equation}
M_n(t) \coloneqq \int_{\R_+} v^n\, f(v,t) \, {\rm d} v,
\end{equation}
then the space $\mathcal{P}_s(\R_+)$ contains probability densities on $\R_+$ having bounded moments up to order $s$.
We now state some elementary observations regarding solutions of the $\varepsilon$-BDY PDE \eqref{eq:epsilon_BDY}.
\begin{lemma}\label{lem:epsilon_BDY}
Let $\varepsilon > 0$ be fixed. Assume that $f_\varepsilon$ is a solution to the $\varepsilon$-BDY PDE \eqref{eq:epsilon_BDY} or equivalently to its weak form~\eqref{eq:Boltz.coll}-\eqref{Boltz2}, starting from an initial condition $f^0_\varepsilon \in \mathcal{P}_1(\mathbb{R}_+)$, which is a smooth probability density with unitary mass and a prescribed mean value $\mu > 0$. Then, the $\varepsilon$-BDY PDE preserves both the total probability mass and the mean value, i.e.,
\begin{equation*}
\frac{\dd}{\dd t} \int_0^\infty f_\varepsilon(v,t)\,\dd v = 0 \quad \textrm{and} \quad \frac{\dd}{\dd t} \int_0^\infty v\,f_\varepsilon(v,t)\,\dd v = 0.
\end{equation*} 
\end{lemma}
\begin{proof}
Straightforward computations (by setting $\varphi(v) \equiv 1$ and $\varphi(v) = v$ respectively in~\eqref{Boltz2}) show that the zero-th and first moments of $f_\varepsilon$ are conserved in time. 
\end{proof}

\begin{remark}
Only the first moment $M^\varepsilon_1$ can be uniformly bounded, as it is actually constant in time. We cannot state anything in general about higher order moments.
Let us consider for example the second moment $M_2^\varepsilon$, whose evolution is given by
\begin{equation}\label{eq_M2eps}
\dfrac{\dd}{\dd t}M_2^\varepsilon(t) = \lambda\, \varepsilon^2\, r[f_\varepsilon](t) +\lambda\, \varepsilon\,\left(\mu\, r[f_\varepsilon](t)-r_1[f_\varepsilon](t)\right),
\end{equation}
where
\begin{equation}\label{eq:proportion_of_rich_1}
r_1[f_\varepsilon](t) \coloneqq \int_{\varepsilon}^\infty v\, f_\varepsilon(v,t)\,\dd v,
\end{equation}
which is a non-negative quantity. Therefore, as is expected from the microscopic dynamics, $M_2^\varepsilon$ is not conserved in general for every $\varepsilon >0$, and a-priori it might be decreasing or increasing.
\end{remark}

\begin{remark}
    In general, the $\varepsilon$-BDY PDE~\eqref{eq:epsilon_BDY} does not have a unique equilibrium. In order to see this, it is enough to consider a smooth function $h:[0,\varepsilon] \to \R_+$ which satisfies $\int_0^\varepsilon h(z) \, {\rm d} z =1$. Then, we can construct an equilibrium distribution $f$ simply by copying and scaling the shape of $h$. In fact, one can define $f$ starting from $h$ on a discrete lattice of size $\varepsilon$ by setting $f(z + n\,\varepsilon) := h(z)\,(1-r)\,r^n$, where $z \in [0, \varepsilon),\, n \in \mathbb{N}$. Since $h$ can have an arbitrary form and it shapes the equilibrium within each $\varepsilon$ interval of the lattice, the steady-state is not unique. Roughly speaking, the reason why the equilibrium is not unique (and thus the equilibrium distribution will depend on the choice of the initial datum) is that the operator defined by the right-hand side of the $\varepsilon$-BDY PDE is locally blind, as it can only move mass by exactly a distance $\varepsilon$. We remark that the Boltzmann-Gibbs (exponential) distribution defined by
   \begin{equation}\label{eq:BG}
f^\infty(v) = \frac{\expo^{-\frac{v}{\mu}}}{\mu}~~\textrm{for each $v \in \mathbb{R}_+$},
\end{equation}
is a possible equilibrium for all $\varepsilon > 0$. In general, any function of the form
\[
f(v) = C\, \expo^{-a v}\,q(v), \quad \text{where} \quad  q(v+\varepsilon)=q(v)
\]
is a positive, non-constant and $\varepsilon$-periodic function, also satisfies  $f(v+\varepsilon)=\expo^{-a \varepsilon}\,f(v)$ and is a possible stationary solution after  normalization.
\end{remark}

Although the primary focus of this manuscript is not on the detailed analysis of the $\varepsilon$-BDY PDE \eqref{eq:epsilon_BDY}, it provides a crucial starting point for the rigorous derivation of the BDY PDE \eqref{eq:mainPDE}. Indeed, the main idea is to send $\varepsilon \to 0$ and show that an appropriately scaled version of the solution $f_{\varepsilon}$ to \eqref{eq:epsilon_BDY} converges to the solution of the BDY PDE \eqref{eq:mainPDE}. Formally it is possible to see the following fact.

\begin{proposition}[Formal asymptotic PDE as $\varepsilon \to 0$]\label{prop:quasi-invariant}
Assume that $f_\varepsilon$ is a classical solution to the $\varepsilon$-BDY PDE \eqref{eq:epsilon_BDY}, with $\lambda =2$ and the initial condition $f^0_\varepsilon \in \mathcal{P}_1(\mathbb{R}_+)$ being a smooth probability density with a prescribed mean value $\mu > 0$. Then as $\varepsilon \to 0$, $f_\varepsilon(v,t/\varepsilon^2)$ converges to the solution $f$ of the BDY PDE \eqref{eq:mainPDE}.
\end{proposition}

\begin{proof}
For $0 < \varepsilon \ll 1$, we have
\begin{equation}\label{eq:Taylor1}
f_\varepsilon(v\pm \varepsilon,t) = f_\varepsilon(v,t) \pm \varepsilon\,\partial_v f_\varepsilon(v,t) + \frac{\varepsilon^2}{2}\,\partial_{vv} f_\varepsilon(v,t) + \mathcal{O}(\varepsilon^3).
\end{equation}
Therefore, thanks to the approximation~\eqref{eq:Taylor2}, for $v\geq \varepsilon$, the right side of \eqref{eq:epsilon_BDY} can be well-approximated by
\begin{equation*}
\begin{aligned}
&f_\varepsilon(v,t)+\varepsilon\,\partial_v f_\varepsilon(v,t)+\frac{\varepsilon^2}{2}\,\partial_{vv} f_\varepsilon(v,t)\\[0.2cm]
&\phantom{\qquad \qquad }+\left(1-\varepsilon\,f_\varepsilon(0,t)\right)\left(f_\varepsilon(v,t)-\varepsilon\,\partial_v f_\varepsilon(v,t) + \frac{\varepsilon^2}{2}\,\partial_{vv} f_\varepsilon(v,t)\right)+ \mathcal{O}(\varepsilon^3)\\[0.2cm]
&\qquad \qquad \qquad \qquad \qquad \qquad \qquad \qquad \qquad \qquad -f_\varepsilon(v,t)-\left(1-\varepsilon\,f_\varepsilon(0,t)\right)f_\varepsilon(v,t)\\[0.2cm]
&= \varepsilon^2\,\left(\partial_{vv} f_\varepsilon(v,t) + f_\varepsilon(0,t)\,\partial_v f_\varepsilon(v,t)\right)+ \mathcal{O}(\varepsilon^3).
\end{aligned}
\end{equation*}
On the other hand, for $0\leq v < \varepsilon$, the right side of \eqref{eq:epsilon_BDY} can be approximated by
\begin{equation*}
\begin{aligned}
&f_\varepsilon(v,t)+\varepsilon\,\partial_v f_\varepsilon(v,t)+\frac{\varepsilon^2}{2}\,\partial_{vv} f_\varepsilon(v,t) - \left(1-\varepsilon\,f_\varepsilon(0)\right)f_\varepsilon(v,t)\\[0.2cm]
&= \varepsilon\,\left(\partial_v f_\varepsilon(v,t) + f_\varepsilon(0,t)\,f_\varepsilon(v,t)\right) + \mathcal{O}(\varepsilon^2).
\end{aligned}
\end{equation*}
Consequently, as $\varepsilon \to 0$, $f_\varepsilon(v,t/\varepsilon^2)$ converges to the solution $f$ of the following nonlinear PDE with a nonlinear Robin-type boundary condition
\begin{equation*}
\begin{cases}
\partial_t f(v,t) = \partial_{vv} f(v,t) + f(0,t)\,\partial_v f(v,t), &~~v > 0,~t\geq 0,\\
\partial_v f(v,t) + f(0,t)\,f(v,t) = 0, &~~ v = 0,~t\geq 0.
\end{cases}
\end{equation*}
This completes the proof of Proposition \ref{prop:quasi-invariant}.
\end{proof}

Building upon the previous formal result, as we want to derive the BDY-PDE from the $\varepsilon$-BDY-PDE, we discuss some regularity properties of the $\varepsilon$-BDY-PDE.

\subsection{Existence and uniqueness of the solution of the $\varepsilon$-BDY PDE}

 We now address the existence and uniqueness of the solution to the initial-value problem
\begin{equation}
	\begin{cases}
    \partial_tf_\varepsilon=Q_\varepsilon(f_\varepsilon,f_\varepsilon), & t>0 \\
		f_\varepsilon(v,0)=f_\varepsilon^0(v),
	\end{cases}
	\label{eq:IVP}
\end{equation}
where the collisional operator $Q_\varepsilon
$ is defined in~\eqref{eq:epsilon_BDY} and $f_\varepsilon^0$ is a prescribed initial datum.

Let us denote $\mathcal{P}(\R_+)$ to be the space of absolutely continuous probability densities on $\mathbb R_+$ and set 
\begin{equation}
    X \coloneqq \mathcal{C}([0,T];\cP(\R_+))
\end{equation}
the space of the continuous (in time) mappings from $[0,T]$ for some $T > 0$, to $\cP(\R_+)$, that we endow with the $L^1$ norm (every element of this space has a unitary $L^1$ norm). As $L^1(\R_+)$ is a Banach space, $(\cP(\R_+),||\cdot ||_{L^1})$ is a complete metric space (but it cannot be a Banach space because it is not a vector space). In fact, $\cP(\R_+)$ is a subset of $L^1(\mathbb R_+)$ that is a Banach space; it suffices to show that it is closed. Let $(f_n)_{n\geq 1} \in \cP(\R_+)$ such that $f_n\to f$ in $L^1$. Then there exists a subsequence of $(f_n)_{n\geq 1}$ which converges a.e. to $f$, thus $f\ge0$ a.e. as well. Since $g\mapsto\int g$ is a bounded linear functional on $L^1$, $\int f=\lim_n\int f_n=1$. Hence $f\in\cP(\R_+)$, and $\cP(\R_+)$ is closed, thus complete.

If $f\in X$, we write $f(t)$ to mean the probability measure $f(\cdot,t)$.
Given $f,\,g\in X$, we define the mapping $\varrho:X\times X\to\R_+$
$$ \varrho(f,g):=\sup_{t\in [0,\,T]}||f(t)-g(t)||_{L^1(\R_+)}, $$
which is a metric in $X$. 
Furthermore, the metric space $(X,\,\varrho)$ is complete.

We can then prove the following result.
\begin{theorem} \label{theo:exists_unique}
Let $f_\varepsilon^0\in\cP_1(\R_+)$ with first moment $\mu>0$ and $T>0$ be a final time. Then the initial-value problem~\eqref{eq:IVP} has a unique solution $f_\varepsilon\in \mathcal{C}([0,T];\cP(\R_+))$ which has a constant first moment $\mu$.
\end{theorem}
\begin{proof}
We shall prove existence and uniqueness by applying Banach's fixed-point theorem to the complete metric space $X$. First, let us introduce the following positive and increasing function:
\begin{equation}
    A_{f_\varepsilon}(t) \coloneqq \dfrac{\lambda}{2}\,\int_0^t  \left[r[f_\varepsilon](\tau)+1\right] \, {\rm d}\tau.
\end{equation} 
It is readily verified that
\begin{equation}\label{def:Aprime}
    A'_{f_\varepsilon}(t) = \dfrac{\lambda}{2}\, [1+r[f_\varepsilon](t)] \in \left[\frac{\lambda}{2},\lambda\right].
\end{equation}
Then we remark that the formal solution of~\eqref{eq:epsilon_BDY} is given by the Duhamel formula
\begin{align}
	\begin{aligned}[b]
		f_\varepsilon(v,t) &= \expo^{-A_{f_\varepsilon}(t)}\,f_\varepsilon^0(v) +\dfrac{\lambda}{2}\int_0^t e^{-(A_{f_\varepsilon}(t)-A_{f_\varepsilon}(\tau))}\,\mathcal{T}_{f_\varepsilon}(v,\tau) \, \dd \tau,
	\end{aligned}
	\label{eq:Boltztype.mild}
\end{align}
where
\begin{equation}
   \mathcal{T}_{f_\varepsilon}(v,\tau) \coloneqq \left(f_\varepsilon(v+\varepsilon,\tau)  +f_\varepsilon(v,\tau)\,\mathbbm{1}_{v<\varepsilon}+r[f_\varepsilon](\tau)\,f_\varepsilon(v-\varepsilon,\tau)\,\mathbbm{1}_{v\geq \varepsilon}\right).
\end{equation}
Equation~\eqref{eq:Boltztype.mild} can be viewed as a fixed-point equation
\[f_\varepsilon=\cQ_\varepsilon(f_\varepsilon),\]
where $\cQ_\varepsilon$ is the operator defined on $X$ whose weak form is given by
\begin{align}
	\begin{aligned}[b]
		\int_{\R_+}\varphi(v)\,\cQ_\varepsilon(f_\varepsilon)(v,t)\,\dd v &= e^{-A_{f_\varepsilon}(t)}\int_{\R_+}\varphi(v)\,f_\varepsilon^0(v)\,\dd v  +\dfrac{\lambda}{2}\int_0^t e^{-(A_{f_\varepsilon}(t)-A_{f_\varepsilon}(\tau))}\int_{\R_+} \varphi(v)\,\mathcal{T}_{f_\varepsilon}(v,\tau)\, {\rm d} v\, \dd \tau.
	\end{aligned}
	\label{eq:cQ.weak}
\end{align}
Proving the existence and uniqueness of~\eqref{eq:IVP} with~\eqref{eq:epsilon_BDY} amounts to show that $\cQ_\varepsilon$ admits a unique fixed point in $X$. For this purpose, we verify the assumptions of Banach's fixed-point theorem.
\begin{enumerate}
\item First, we show that $\cQ_\varepsilon$ maps $X$ into itself. To this aim, let $f_\varepsilon\in X$, then:
\begin{enumerate}[label=(\roman*)]
\item The non-negativity of  $\cQ_\varepsilon(f_\varepsilon)(t)$ is clear from~\eqref{eq:Boltztype.mild}.

We now remark that
\begin{equation}\label{norm.Tau}
\|\mathcal{T}_{f_\varepsilon}(\cdot,\tau) \|_{L^1(\R_+)} =r[f_\varepsilon](\tau)+1 \in [1,2] \quad \text{for all $\tau \geq 0$}.
\end{equation}
Setting $\varphi\equiv 1$ in~\eqref{eq:cQ.weak} gives
$$ \int_{\R_+}\cQ_\varepsilon(f_\varepsilon)(v,t)\,\dd v=\expo^{-A_{f_\varepsilon}(t)}+\frac{\lambda}{2}\,\expo^{-A_{f_\varepsilon}(t)}\int_0^t \expo^{A_{f_\varepsilon}(\tau)}\,(r[f_\varepsilon](\tau)+1)\,\dd \tau=1, $$
in view of~\eqref{def:Aprime}.
As a consequence, $\cQ_\varepsilon(f_\varepsilon)(t)$ is a probability measure for $t>0$.
Therefore, we conclude that $\cQ_\varepsilon(f)(t)\in \cP(\R_+)$ whenever $f(t)\in \cP(\R_+)$.
\item As for the continuity of the mapping $t\mapsto\cQ_\varepsilon(f_\varepsilon)(t)$, we fix $t_0\in (0,\,+\infty)$ and check that $||\cQ_\varepsilon(f_\varepsilon)(t)-\cQ_\varepsilon
(f_\varepsilon)(t_0)||_{L^1(\R_+)}\to 0$ when $t\to t_0$. For the sake of simplicity and without loss of generality, we assume $t>t_0$. 
We observe that
\begin{align*}
&||\cQ_\varepsilon(f_\varepsilon)(t)-\cQ_\varepsilon
(f_\varepsilon)(t_0)||_{L^1(\R_+)} \\ 
&\le \left|\expo^{-A_{f_\varepsilon}(t)}-\expo^{-A_{f_\varepsilon}(t_0)}\right| \left[ \| f_\varepsilon^0\|_{L^1(\R_+)}+ \dfrac{\lambda}{2} \int_0^{t_0}\expo^{A_{f_\varepsilon}(\tau)}\,
\|\mathcal{T}_{f_\varepsilon}(\cdot,\tau) \|_{L^1(\R_+)} \, \dd \tau\right]\\
&\phantom{=}+\dfrac{\lambda}{2}\,\expo^{-A_{f_\varepsilon} (t)} \int_{t_0}^t \expo^{A_{f_\varepsilon}(\tau)}\, \|\mathcal{T}_{f_\varepsilon}(\cdot,\tau)\|_{L^1(\R_+)} \, \dd\tau\\
& \le \left| A_{f_\varepsilon(t)}-A_{f_\varepsilon}(t_0)\right| \left[ \| f_\varepsilon^0\|_{L^1(\R_+)}+ \dfrac{\lambda}{2} \int_0^{t_0} \expo^{A_{f_\varepsilon}(\tau)}\,
[r[f_\varepsilon](\tau)+1] \, \dd \tau\right]\\
&\phantom{=}+\dfrac{\lambda}{2} \expo^{-A_{f_\varepsilon} (t)} \int_{t_0}^t \expo^{A_{f_\varepsilon}(\tau)} \,[r[f_\varepsilon](\tau)+1]\, \dd\tau\\
&=\left| A_{f_\varepsilon}(t)-A_{f_\varepsilon}(t_0)\right|\, \left[1+ \expo^{A_{f_\varepsilon}(t_0)}-1
\right] + \expo^{-A_{f_\varepsilon} (t)}\, \left[\expo^{A_{f_\varepsilon} (t)}-\expo^{A_{f_\varepsilon} (t_0)}\right]
\end{align*}

where we have used the elementary inequality $|\expo^{-a}-\expo^{-b}| \le |a-b|$ for all $a,b >0$, and \eqref{norm.Tau} in the second inequality,
and integration by parts in the last equality along with $ \| f_\varepsilon^0\|_{L^1(\R_+)}=1$. We also notice that
\[
\left| A_{f_\varepsilon}(t)-A_{f_\varepsilon}(t_0) \right| = \dfrac{\lambda}{2}\left(\int_{0}^t[r[f_\varepsilon](\tau)+1] \, \dd \tau-\int_0^{t_0} [r[f_\varepsilon](\tau)+1] \, \dd \tau\right) = \dfrac{\lambda}{2}\int_{t_0}^t[r[f_\varepsilon](\tau)+1] \, \dd \tau
\]
we can conclude by the boundedness of $r[f_\varepsilon] \le 1$, and by the continuity of the integral function and of the exponential function as $t \to t_0$.
\end{enumerate}
\item Finally, to see that $\cQ_\varepsilon$ is a contraction on $X$ we consider $f_\varepsilon^1,\,f_\varepsilon^2\in X$ and estimate
\begin{equation}
\begin{aligned}[b]\label{proof.contr}
    \|  \cQ_\varepsilon(f_\varepsilon^1)(t)-\cQ_\varepsilon(f_\varepsilon^2)(t)\|_{L^1(\R_+)} &\le \| f_\varepsilon^0\|_{L^1(\R_+)} \left| A_{f^1_\varepsilon}(t)-A_{f^2_\varepsilon}(t) \right|\\
    &\phantom{\le}+\dfrac{\lambda}{2}\,\int_0^t \left|\expo^{-(A_{f^1_\varepsilon}(t)-A_{f^1_\varepsilon}(\tau))}-\expo^{-(A_{f^2_\varepsilon}(t)-A_{f^2_\varepsilon}(\tau))}\right| [r[f^1_\varepsilon](\tau)+1] \, \dd \tau\\
    &\phantom{\le}+\dfrac{\lambda}{2}\int_0^t \expo^{-(A_{f^2_\varepsilon}(t)-A_{f^2_\varepsilon}(\tau))}\,\|\mathcal{T}_{f^1_\varepsilon}(\cdot,\tau)-\mathcal{T}_{f^2_\varepsilon}(\cdot,\tau)\|_{L^1(\R_+)} \, \dd \tau
\end{aligned}
\end{equation}
where we have used~\eqref{norm.Tau} in the second term.
We now observe that
\begin{equation}\label{A.Lipsch}
\left| A_{f^1_\varepsilon}(t)-A_{f^2_\varepsilon}(t) \right| \leq \dfrac{\lambda}{2}\int_{0}^t \left|r[f^1_\varepsilon](\tau)-r[f^2_\varepsilon]\right|(\tau) \, \dd \tau \le \dfrac{\lambda}{2}\,t\,\sup_{0\leq \tau\leq t} \|f^1_\varepsilon(\tau)-f^2_\varepsilon(\tau)\|_{L^1(\R_+)},
\end{equation}
where we have used the fact that
\begin{equation}\label{r.Lipsch}
\left|r[f^1_\varepsilon](\tau)-r[f^2_\varepsilon](\tau)\right| \le\int_\varepsilon^\infty \left| f_\varepsilon^1(v,\tau)-f^2_\varepsilon(v,\tau)\right| \, \dd v \le \| f_\varepsilon^1(\tau)-f_\varepsilon^2(\tau) \|_{L^1(\R_+)}.
\end{equation}
Next, we prove the Lipschitz continuity of $\mathcal{T}_{f_\varepsilon}$ in the $L^1(\R_+)$ norm. We have 
\[
\begin{split}
\|\mathcal{T}_{f^1_\varepsilon}(\cdot,t)-\mathcal{T}_{f^2_\varepsilon}(\cdot,t) \|_{L^1(\R_+)} &\le \int_{\R_+} \left|f_\varepsilon^1(v,t)-f^2_\varepsilon(v,t)\right| \, \dd v\\
&\phantom{=}+ r[f_\varepsilon^1](t)\int_{\R_+} \left| f_\varepsilon^1(v,t)- f_\varepsilon^2(v,t) \right| \, \dd v + \left|r[f^1_\varepsilon](t)-r[f^2_\varepsilon](t) \right|\int_{\R_+} f_\varepsilon^2 \, \dd v \\
&\le 3\,\|f_\varepsilon^1(t)-f_\varepsilon^2(t) \|_{L^1(\R_+)},
\end{split}
\]
in which we have employed~\eqref{r.Lipsch}, the bound $r[f_\varepsilon^1] \le 1$, and the fact that $\| f_\varepsilon^1\|_{L^1(\R_+)} =\| f_\varepsilon^2\|_{L^1(\R_+)}=1$.  

In order to deal with the second term in~\eqref{proof.contr}, we again exploit the estimate $|\expo^{-a}-\expo^{-b}| \le |a-b|$ for all $a,b >0$, which is applied with $a=A_{f^1_\varepsilon}(t)-A_{f^1_\varepsilon}(\tau)$ and $b= A_{f^2_\varepsilon}(t)-A_{f^2_\varepsilon}(\tau)$, both of which are positive for $t>\tau$. We then obtain
\[
\left|\expo^{-(A_{f^1_\varepsilon}(t)-A_{f^1_\varepsilon}(\tau))}-\expo^{-(A_{f^2_\varepsilon}(t)-A_{f^2_\varepsilon}(\tau))} \right| \le \left|A_{f^1_\varepsilon}(t)-A_{f^2_\varepsilon}(t) \right| +\left| A_{f^1_\varepsilon}(\tau)-A_{f^2_\varepsilon}(\tau) \right|.
\]
In conclusion, we obtain
\begin{align*}
\| \cQ_\varepsilon(f_\varepsilon^1)(t)-\cQ_\varepsilon(f_\varepsilon^2)(t)\|_{L^1(\R_+)} &\le \dfrac{\lambda}{2}\,\left(t +2\,\dfrac{\lambda}{2}\,t^2 \right)\sup_{0\leq \tau\leq t}\|f_\varepsilon^1(\tau)-f_\varepsilon^2(\tau)\|_{L^1(\R_+)} \\
&\quad + \dfrac{\lambda}{2}\int_0^t \left(2\,\dfrac{\lambda}{2}\,\tau+3\right)\,\sup_{0\leq s\leq \tau}\|f_\varepsilon^1(s)-f_\varepsilon^2(s)\|_{L^1(\R_+)} \, \dd \tau.
\end{align*}
where we have used $\| f_\varepsilon^0\|_{L^1(\R_+)}=1$ and \eqref{A.Lipsch} for the first term in the right-hand side of~\eqref{proof.contr}, \eqref{A.Lipsch} twice and $r[f_\varepsilon^1] \le 1$ in the second term, and $\expo^{-(A_{f^2_\varepsilon}(t)-A_{f^2_\varepsilon}(\tau))} \le 1$ for $t >\tau$ and~\eqref{norm.Tau} in the last term.
Taking the supremum over $t \in [0,T]$ gives rise to
\[
\rho(\cQ_\varepsilon(f_\varepsilon^1),\cQ_\varepsilon(f_\varepsilon^2)) \le  \dfrac{\lambda}{2}\,\left(4\,T+\frac 32\,\lambda\, T^2\right)\,\rho(f_\varepsilon^1,f_\varepsilon^2).
\]
The latter is a strict contraction if 
\[
0<T < \dfrac{2\,\sqrt{7} - 4}{3\,\lambda}.
\]
 Finally, a continuation argument can be applied to obtain the result over the entire interval $[0,T]$ for every finite $T>0$.
\end{enumerate}
This concludes the proof of the existence and uniqueness of the solution in the complete metric space $X$. As already remarked, since $(\mathcal P(\mathbb R_+),||\cdot ||_{L^1})$ is closed in $X$ and thus is complete, in view of the fact that $\mathcal{Q}_\varepsilon$ preserves positivity, we deduce that the unique solution is in the complete metric space $(\mathcal P(\mathbb R_+),||\cdot ||_{L^1})$.  
Eventually, we prove that this unique solution has a conserved first moment by construction. Let $f_\varepsilon \in\cP_1$ with first moment $\mu$. Setting $\varphi(v)=v$ in~\eqref{eq:cQ.weak} yields
\begin{equation*}
 \int_{\R_+} v\,\cQ(f_\varepsilon)(v,t)\,\dd v=\expo^{-A_{f_\varepsilon}(t)}\,\mu+\frac{\lambda}{2}\,\expo^{-A_{f_\varepsilon}(t)}\int_0^t \expo^{A_{f_\varepsilon}(\tau)}\,\left[\mu-\varepsilon\, r[f_\varepsilon]+r[f_\varepsilon]\,(\mu+\varepsilon)\right]\,\dd\tau =\mu,
\end{equation*}
again using~\eqref{def:Aprime}. This concludes the proof.
\end{proof}

\begin{remark}
    As the unique solution belongs to $(\mathcal P(\mathbb R_+),||\cdot ||_{L^1})$ and has a constant first moment, it actually lives in the space $(\mathcal P_1(\mathbb R_+),||\cdot ||_{L^1})$.
\end{remark}

\begin{remark}
    A popular family of metrics, which is widely used to study kinetic models, is the so-called Fourier-based distances, which were introduced in~\cite{gabetta_1995} in order to study the trend to equilibrium of solutions to the space homogeneous Boltzmann equation for Maxwell molecules, and then applied to a variety of other kinetic models of Maxwell type~\cite{carrillo_2007} also for socio-economic phenomena~\cite{cao_2026_fractal,torregrossa_wealth_2017, torregrossa_wealth_2018, fraia_2026}. It allows the proof of existence and uniqueness of the solution, as well as of convergence to equilibrium. 
    We want to give here a hint on why this family of metrics does not work for the collision-like kinetic model~\eqref{eq:epsilon_BDY}. 
    
    Given $s>0$ and two probability densities $f_1$ and $f_2$ on $\R$, their Fourier-based distance $d_s(f_1,f_2)$ of order $s$ is defined by
\begin{equation}\label{def:fourier.dist}
d_s(f_1,f_2) \coloneqq \displaystyle \sup\limits_{\xi \neq 0} \dfrac{|\hat{f}_1(\xi)-\hat{f}_2(\xi)|}{|\xi|^s},
\end{equation}
where $\hat{f}_i(\xi,t)\coloneqq \int_{\mathbb R} \expo^{-i\,v\,\xi}\,f_i(v,t)\,\dd v$ denotes the Fourier transform of $f_i(\cdot, t) \in \cP(\R)$ for all $t>0$, $i=1,2$.
The Fourier distance can be defined in the present case on $\R_+$, as the support of $f$ is in $\R_+$. This can be rigorously justified by exploiting the positivity of the microscopic rule~\eqref{micro.rule} on positive interaction wealth, i.e., $v,w \ge 0 \implies v',w' \ge0$. Invoking the same argument as in~\cite{loy_essentials_2025}, elementary computations show that even if the microscopic dynamics was defined on $\R$ but the support of the initial condition $f^0$ was in $\R_+$, then the support of $f(\cdot,t)$ would be preserved in $\R_+$ for all times. 

Fundamental properties of the Fourier distance are:
\begin{enumerate}[label=(\roman*)]
\item The distance $d_s$ is finite for two measures in $\cP_s(\R)$ if they have equal moments up to the order $s-1$ or $[s]$ (the integer part of $s$), depending on whether $s$ is an integer or not. 
\item For $\alpha>0$, the subset $\cP_{s,\alpha,\cM_{s+\alpha}}(\R)$ of $\cP_{s+\alpha}(\R_+)$ which consists of probability measures $f$ with prescribed moments $M_n$ for $n\leq [s]$, and such that $\int_{\R_+}\abs{v}^{s+\alpha}\,f(v,t) \, \dd v$ is uniformly bounded with respect to $f(\cdot,t)$ by a constant $\cM_{s+\alpha}>0$. Then $\cP_{s,\alpha,\cM_{s+\alpha}}(\R)$ endowed with $d_s$ is a complete metric space. 
\end{enumerate}
In the present case, it is worth noticing that $M_1$ is conserved, so that $\cM_1 =M_1^0=\mu$. As a consequence, all distances $d_s$ up to $s=2$ are well defined. However, as the equilibrium is not unique, and it is not uniquely given by the Boltzmann-Gibbs distribution, we cannot state that the second moment is uniformly bounded via ~\eqref{eq_M2eps}. As a consequence, we can consider $s \in (0,1)$  and $\alpha=1-s$, so that $\cM_{s+\alpha}=\cM_1$.
For $T>0$ and $s\in (0,1)$ on the space
$$ X_s:=C^0([0,\,T];\cP_{s,1-s,\cM_1}(\R)) $$
of continuous mappings from $[0,\,T]$ to $\cP_{s,1-s,\cM_1}(\R)$. Given $f,\,g\in X$ one can define the mapping $\varrho_s:X_s\times X_s\to\R_+$
$$ \varrho_s(f,g):=\sup_{t\in [0,\,T]}d_s(f(t),g(t)), $$
which is a metric in $X_s$, and the metric space $(X_s,\,\varrho_s)$ is complete.

However, if one aims to prove the existence and uniqueness by exploiting Banach's fixed-point theorem in this complete metric space (see e.g. \cite{loy_essentials_2025, fraia_2026}) one fails in an attempt to show that $\cQ_\varepsilon$ is a contraction.
In fact, when computing the Fourier-transform $\widehat{\cQ_\varepsilon}$ of the operator $\cQ_\varepsilon$, one is led to evaluate the following boundary layer term
\[
\int_{\R_+} \mathbbm{1}_{v \le \varepsilon}\,(f_\varepsilon^1(v,t)-f_\varepsilon^2(v,t))\,\expo^{-i\,\xi\, v} \,\dd v = \dfrac{1}{2\,\pi}\widehat{ \mathbbm{1}_{v \le \varepsilon}}* (\hat{f}_\varepsilon^1-\hat{f}_\varepsilon^2), \quad \text{with} \quad \widehat{ \mathbbm{1}_{v \le \varepsilon}}=\dfrac{\expo^{-i\,\xi\, \varepsilon}-1}{-i\,\xi}.
\]
As a consequence, our computations show that there exists a constant $C>0$ which depends on $s,\varepsilon, T$ such that 
\[
\rho_s(\widehat{\cQ(f_\varepsilon^2)}(\xi,t),\widehat{\cQ(f_\varepsilon^1)}(\xi,t)) \leq C(\varepsilon,s,T)\, \rho_s(f_\varepsilon^2,f_\varepsilon^1)\left[1+\left\| \dfrac{\expo^{-i\,\xi\ \varepsilon}-1}{|\xi|}\right\|_{L^1(\R)}\right].
\]
The right hand side cannot be bounded, as for high frequencies $\frac{|\expo^{-i\,\xi\, \varepsilon}-1|}{|\xi|} \sim \mathcal{O}(|\xi|^{-1})$ for $\xi \to \infty$, even when $\varepsilon \to 0$.

This issue cannot be overcome to the best of our knowledge, and it is deeply related to the structure of the microscopic agent-based dynamics, in which agents exchange exactly $\varepsilon$ amount of wealth as soon as they possess at least $\varepsilon$ wealth. Thus the discrete BDY dynamics leads to the appearance of the indicator function, which makes the model 'blind' as it is localized in $\varepsilon$. It is the same cause of the non-uniqueness of the equilibrium for the $\varepsilon$-BDY PDE, and this issue prevents us from establishing uniform bounds for higher-order moments (beyond the first).

The presence of the indicator function $\mathbbm{1}_{v \le \varepsilon}$ which is related to the boundary layer and the local ``blind'' microscopic dynamics, also hampers the proof of standard $L^r$ or $H^r$ a-priori regularity~\cite{loy_essentials_2025}.
\end{remark}

\section{Derivation and analysis of the BDY-PDE}\label{sec:sec2}

\subsection{Derivation by means of the quasi-invariant limit}
An alternative and more rigorous approach for the derivation of the BDY PDE~\eqref{eq:mainPDE}, which is a nonlinear Fokker-Planck equation, is the \emph{quasi-invariant limit} technique~\cite{toscani_2006,toscani_1999}. The quasi-invariant limiting procedure allows us to show that a subsequence of $f_\varepsilon$, which is a solution of a collision-like kinetic equation \eqref{eq:Boltz.coll}-\eqref{Boltz1}-\eqref{kernel}, converges (on a suitable time scale) to a solution of the Fokker-Planck type equation \eqref{eq:mainPDE}. This technique essentially relies on considering binary interactions which produce small changes and to analyze them on an appropriate long time scale which compensates the smallness of the interaction. In the microscopic rule~\eqref{micro.rule}, there is actually a small ``quasi-invariant'' exchange of money, which we can observe on a slower time scale in order to observe many small interactions accumulating. In particular, the proof of Proposition~\ref{prop:quasi-invariant} suggests that the appropriate time scale is $t/\varepsilon^2$.  In particular, instead of considering $f_\varepsilon(v,t/\varepsilon^2)$, we consider the equivalent high frequency regime/scaling defined by
\begin{equation}\label{def:freq.eps}
\lambda \longmapsto \dfrac{\lambda}{\varepsilon^2}.
\end{equation}
We now aim to derive by means of the quasi-invariant limit the nonlinear Fokker-Planck equation~\eqref{eq:mainPDE} from the collision-like description~\eqref{eq:Boltz.coll}-\eqref{Boltz1}. Following the procedure presented, e.g., in~\cite{torregrossa_wealth_2017}, we start from~\eqref{Boltz2}, which now reads as
\begin{equation}\label{Boltz2.resc}
\begin{split}
\dfrac{1}{\varepsilon^2}\left\langle Q_\varepsilon(f_\varepsilon,f_\varepsilon),\varphi\right\rangle &= \dfrac{\lambda}{2\,\varepsilon^2}\int_{\R_+} \int_{\R_+} \left[\varphi(v-\varepsilon)-\varphi(v)\right]\,\mathbbm{1}_{v\geq \varepsilon}\, f_\varepsilon(v,t)\,f_\varepsilon(w,t)\,{\rm d} v\, {\rm d} w\\
&+\dfrac{\lambda}{2\,\varepsilon^2}\int_{\R_+} \int_{\R_+} \left[\varphi(w+\varepsilon)-\varphi(w)\right]\,\mathbbm{1}_{v\geq \varepsilon}\, f_\varepsilon(v,t)\,  f_\varepsilon(w,t) \, {\rm d}v {\rm d} w.
\end{split}
\end{equation}
For $m \in \mathbb{N}_+$, let $\mathcal{C}^m(\R_+)$ be the set of $m$ times continuously differentiable functions, endowed with its natural norm $||\cdot||_m$:
\[
{\displaystyle \|\varphi\|_m} \coloneqq {\begin{cases}\sup\limits_{v \in \R_+ }|\varphi(v)|&{\text{ if }}~m=0,\\ \|\varphi\|_0+\displaystyle\sum _{k=1}^{m}\left\| \dfrac{
\dd^k \varphi}{\dd v^k}\right\|_0 &{\text{ if }}~m \geq 1.\end{cases}}
\]
As the domain $\R_+$ is unbounded, we consider the following set of test functions
\begin{equation}\label{def:test.phi}
\mathcal{B}^m\coloneqq \left\{ \varphi \in \mathcal{C}^m(\R_+) \mid ||\varphi||_m \leq 1\right\}.
\end{equation}
In the sequel we shall consider $m=3$.

A Taylor expansion for $\varphi(v\pm \varepsilon)$ around $v$ for small enough $\varepsilon$ leads us to
\begin{equation*}
\begin{split}
\dfrac{1}{\varepsilon^2}\left\langle Q_\varepsilon(f_\varepsilon,f_\varepsilon),\varphi \right\rangle &= \dfrac{\lambda}{2\,\varepsilon^2} \int_{\R_+} \left[-\varepsilon\,\varphi'(v)+\dfrac{\varepsilon^2}{2}\,\varphi''(v)-\dfrac{\varepsilon^3}{6}\varphi'''(\bar{v}^-)\right]\mathbbm{1}_{v\geq \varepsilon}\, f_\varepsilon(v,t)\,{\rm d} v \\
&\phantom{=}+\dfrac{\lambda}{2\,\varepsilon^2}\,r[f_\varepsilon]\,\int_{\R_+} \left[\varepsilon\, \varphi'(w)+\dfrac{\varepsilon^2}{2}\,\varphi''(w)+\dfrac{\varepsilon^3}{6}\,\varphi'''(\bar{w}^+)\right] f_\varepsilon(w,t) \, {\rm d} w,
\end{split}
\end{equation*}
in which $\bar{v}^- = v - \alpha_{-}\,\varepsilon \in (v-\varepsilon, v)$ and $\bar{w}^+ = w + \alpha_{+}\,\varepsilon \in (w,w+\varepsilon)$
for some $\alpha_{\pm}\in (0,1)$.
Then, taking into account equation~\eqref{eq:Taylor2}, we obtain
\begin{equation}\label{Boltz.approx}
\begin{split}
\dfrac{1}{\varepsilon^2}\left\langle Q_\varepsilon(f_\varepsilon,f_\varepsilon),\varphi \right\rangle =& -\dfrac{\lambda}{2}\,f_\varepsilon(0,t)\int_{\R_+} \varphi'(v)\, f_\varepsilon(v,t) \, {\rm d} v +\dfrac{\lambda}{2}\int_{\R_+} \varphi''(v)\, f_\varepsilon(v,t) \, {\rm d} v\\
&+\dfrac{\lambda}{2\,\varepsilon^2}\left( \int_{0}^\varepsilon \varepsilon\, \varphi'(v)\,f_\varepsilon(v,t)\,{\rm d} v - \int_{0}^\varepsilon  \dfrac{\varepsilon^2}{2}\,\varphi''(v)\, f_\varepsilon(v,t)  \, {\rm d} v\right)\\
&+\dfrac{\lambda}{2\varepsilon^2}\left(\int_{\R_+} \left(-\frac{\varepsilon^2}{2}\,\partial_v f_\varepsilon(\bar{v}_\varepsilon,t)\,\varepsilon\,\varphi'(v)-\varepsilon\, f_\varepsilon(0,t)\,\frac{\varepsilon^2}{4}\,\varphi''(v)\right) f_\varepsilon(v,t) \, {\rm d}v  \right.\\
&\left.\phantom{\int_{\R_+} -\varphi'(v)\varepsilon\partial_v f_\varepsilon(\bar{v}_-)}+\dfrac{\varepsilon^3}{6} \left(\varphi'''(\bar{w}^+)-\varphi'''(\bar{v}^-)\right)r[f_\varepsilon]  \right),
\end{split}
\end{equation}
which can be rewritten as
\begin{equation}\label{Boltz.approx2}
\begin{split}
\dfrac{1}{\varepsilon^2}\left\langle Q_\varepsilon(f_\varepsilon,f_\varepsilon),\varphi \right\rangle  =&  -\dfrac{\lambda}{2}\,f_\varepsilon(0,t)\int_{\R_+} \varphi'(v)\, f_\varepsilon(v,t) \, {\rm d} v +\dfrac{\lambda}{2}\int_{\R_+} \varphi''(v)\,f_\varepsilon(v,t) \, {\rm d} v\\
&+\dfrac{\lambda}{2}\,\varphi'(0)\,f_\varepsilon(0,t)+\int_{\R_+}\mathcal{R}_\varepsilon\left(\varphi(v),t\right) f_\varepsilon(v,t) \, {\rm d}v.
\end{split}
\end{equation}
The first two terms on the right-hand side of \eqref{Boltz.approx2} will give rise to the classical drift and diffusion terms of the desired Fokker-Planck equation, while the third term takes into account of the boundary layer, which will allow us to determine the correct boundary conditions, in order to ensure certain conservation properties. The quantity $\mathcal{R}_\varepsilon$ denotes the remainder, which is given by
\begin{equation}\label{def:remainder}
\begin{split}
\int_{\R_+ }\mathcal{R}_\varepsilon\left(\varphi(v),t\right) &f_\varepsilon(v,t)\,{\rm d} v \coloneqq  \\
&\dfrac{\lambda}{2}\left(\frac{\varepsilon}{2}\left(\varphi'(v)\,f_\varepsilon(v,t)\right)'\big\vert_{v=\tilde{v}_\varepsilon}-\int_{0}^\varepsilon  \dfrac{1}{2}\,\varphi''(v)\, f_\varepsilon(v,t)  \, {\rm d} v \right.\\
&\left.-\partial_v f_\varepsilon(\bar{v}_\varepsilon,t)\int_{\R_+} \frac{\varepsilon}{2}\,\varphi'(v)\,f_\varepsilon(v,t) \, {\rm d}v-\int_{\R_+}\varepsilon\, f_\varepsilon(0,t)\,\frac{1}{4}\,\varphi''(v)\,f_\varepsilon(v,t) \, {\rm d}v \right.\\
&+\left.\dfrac{\varepsilon}{6} \left(\varphi'''(\bar{w}^+)-\varphi'''(\bar{v}^-)\right)r[f_\varepsilon]  \right)+\mathcal{O}(\varepsilon^2),
\end{split}
\end{equation}
where $\tilde{v}_\varepsilon \in (0,\varepsilon)$.
It is immediate to verify that
\begin{equation}
\int_{\R_+}\mathcal{R}_\varepsilon\left(\varphi(v),t\right) f_\varepsilon(v,t) \, {\rm d}v \xrightarrow{\varepsilon \to 0^+} 0,
\end{equation}
because
\[
\left|\int_{\R_+}\mathcal{R}_\varepsilon\left(\varphi(v),t\right) f_\varepsilon(v,t) \, {\rm d}v\right| \le \dfrac{\varepsilon\,\lambda}{2}\,||\varphi||_m\, \left(|\partial_v f_\varepsilon(\bar{v}_\varepsilon,t)|+|\partial_v f_\varepsilon(\tilde{v}_\varepsilon,t)|+f_\varepsilon(\tilde{v}_\varepsilon,t)+f_\varepsilon(0,t)+1\right)
\]
where we have employed the elementary fact that $r[f_\varepsilon] \le 1$, and $\varphi\in \mathcal{D}$. If we assume also that $f_\varepsilon$ and $\partial_v f_\varepsilon$ are bounded in a neighborhood of the boundary $v=0$ for all small enough $\varepsilon$ and uniformly in time, then in the limit as $\varepsilon \to 0^+$, the solution $f_\varepsilon$ to~\eqref{eq:Boltz.coll}-\eqref{Boltz1}-\eqref{kernel} under the scaling \eqref{def:freq.eps}, then converges to the solution of the weak Fokker-Planck equation
\begin{equation}\label{eq:FP-weak}
\begin{split}
\int_{\R_+}\varphi(v)\, \partial_t f(v,t)\, {\rm d }v &=  \int_{\R_+}\varphi(v)\, J(f) \, \dd v
\end{split}
\end{equation}
where the operator $J$ is defined by
\begin{equation}\label{def:op.J}
\begin{split}
\int_{\R_+}\varphi(v) J(f) \, \dd v \coloneqq& - \dfrac{\lambda}{2}\,f(0,t)\int_{\R_+} \varphi'(v)\, f(v,t)\,{\rm d} v +\dfrac{\lambda}{2}\int_{\R_+} \varphi''(v)\,f(v,t)\,{\rm d} v \\&+\dfrac{\lambda}{2}\,\varphi'(0)\,f(0,t).
\end{split}
\end{equation}
Integrating by parts yields that
\begin{equation}\label{eq:FP-weak.2}
\begin{split}
\dfrac{\dd}{\dd t} \int_{\R_+} f(v,t)\,\varphi(v) \, {\rm d} v &=  \dfrac{\lambda}{2}\,f(0,t)\int_{\R_+} \varphi(v)\, \partial_v f(v,t)\,{\rm d} v  +\dfrac{\lambda}{2}\int_{\R_+} \varphi(v)\,\partial_{vv}f(v,t) \, {\rm d} v\\
&+\dfrac{\lambda}{2}\left[\varphi'(v)\,f(v,t) -\varphi(v)\,\partial_v f(v,t) - f(0,t)\,\varphi(v)\, f(v,t)\right]\big\vert_{v=0}^\infty\\
& +\dfrac{\lambda}{2}\,\varphi'(0)\,f(0,t),
\end{split}
\end{equation}
where in the second and third line we have several boundary terms. If we choose initial data with a smooth and rapid decay, then various boundary terms evaluated at infinity vanish since $\lim_{v \to \infty} f(v,t)=0$ and $\lim_{v \to \infty} \partial_v f(v,t)=0$. Consequently, we arrive at
\begin{equation}\label{eq:FP-weak.3}
\begin{split}
\int_{\R_+}\varphi(v)\, \partial_t f(v,t)\, {\rm d }v =&  \int_{\R_+}\varphi(v)\left(\dfrac{\lambda}{2}\,f(0,t)\, \partial_v f(v,t)  +\dfrac{\lambda}{2}\,\partial_{vv} f(v,t)\right)\, {\rm d} v \\
&+\dfrac{\lambda}{2}\,\varphi(0)\left[f(0,t)\,f(v,t) + \partial_v f(v,t)\right]\big\vert_{v=0}.
\end{split}
\end{equation}
From the weak form~\eqref{eq:FP-weak.3} of the Fokker-Planck equation, we obtain
the corresponding problem in strong form, which is~\eqref{eq:mainPDE} (with $\lambda = 2$), which we restate here for the reader's convenience:
\begin{equation}\label{eq:mainPDE.bis}
\begin{cases}
\partial_t f(v,t) = \dfrac{\lambda}{2}\left(\partial_{vv} f(v,t) + f(0,t)\,\partial_v f(v,t)\right), &~~v > 0,~t\geq 0,\\
\partial_v f(v,t) + f(0,t)\,f(v,t) = 0, &~~ v = 0,~t\geq 0.
\end{cases}
\end{equation}
In summary, we have proved the following:
\begin{theorem}\label{thm:quasi-invariant-limit}
Let $f^0\in \mathcal{P}_1(\R_+)$ be a smooth probability density, and assume that $f_\varepsilon(\cdot,0) = f(\cdot,0) =:f^0$. Then as $\varepsilon \to 0^+$, the weak solution $f_\varepsilon(v,t)$ to the Boltzmann-type equation \eqref{eq:Boltz.coll}-\eqref{Boltz1}-\eqref{kernel} in the high frequency regime~\eqref{def:freq.eps} converges, under the assumption that $f_\varepsilon$ and $\partial_v f_\varepsilon$ are uniformly bounded in time in a neighbourhood of the boundary $v=0$ for all sufficiently small $\varepsilon$, to a probability density $f(v,t)$ which is a weak solution of the nonlinear Fokker-Planck equation~\eqref{eq:mainPDE}.
\end{theorem}

\subsection{Main properties of the solution}

The boundary value problem~\eqref{eq:mainPDE.bis} was postulated in~\cite{calvez_one_2010}  as a model for cell migration in response to a chemo-attractant, and then further studied in~\cite{calvez_analysis_2012,lepoutre_cell_2014} also for cases in which the mass of the distribution $f$ is not unitary.

For the purposes of the present work which is in the context of kinetic models for econophysics, where $f$ are time dependent measures with unitary mass, we highlight that~\eqref{eq:mainPDE.bis} is a nonlinear Fokker-Planck equation with constant diffusion and nonlinear flux (due to the presence of the flux coefficient $f(0,t)$), and a Robin-type boundary condition. Robin-type boundary conditions in the context of kinetic models for wealth exchange have also appeared, for example, in~\cite{torregrossa_wealth_2017}, where the authors derive by means of a quasi-invariant limit procedure a Fokker-Planck equation with non-constant diffusion. The latter needs to be coupled with a Robin-type condition, which acts as a no-flux condition, which guarantees the conservation of mass, by requiring an exact balance between the advective and diffusive fluxes at the boundary $v=0$.
In the present case, it is useful for the next section to rewrite the PDE \eqref{eq:mainPDE.bis} in the flux form as
\begin{equation}\label{eq:flux}
\partial_t f = \partial_v \mathcal{F}[f], \quad \textrm{with~~$\mathcal{F}[f](\cdot,t) \coloneqq \dfrac{\lambda}{2}\left(\partial_v f(\cdot,t)+f(0,t)\,f(\cdot,t)\right)$}.
\end{equation}
In this way, the boundary condition reads $\mathcal{F}[f](0,t)=0$ and can be viewed as a no-flux condition which leads to preservation of total probability mass.

We stress the fact that the quadratic term $f(0,t)^2$ appearing in the Robin-type boundary condition, which reads as
\begin{equation}\label{eq:Robin-type-BC}
\partial_v f(0,t) + f^2(0,t) = 0,
\end{equation}
comes from the flux coefficient $f(0,t)$, as expected from the pioneering work~\cite{feller_1951}. The latter condition also implies that $\partial_v f(0,t)$ is always negative.

As already mentioned in~\cite{lepoutre_cell_2014}, we recall that if $f^0 \ge 0$, then $f \ge 0$ for all times thanks to the maximum principle.

\begin{remark}
We emphasize that imposing an initial condition $f^0(v) \coloneqq f(v,0)$ compatible with the Robin-type boundary condition is crucial for the BDY PDE \eqref{eq:mainPDE} to preserve the total probability mass. Meanwhile, the Robin-type boundary condition \eqref{eq:Robin-type-BC} must already hold at the initial time $t=0$ in order for classical solutions to the initial boundary value problem \eqref{eq:mainPDE.bis} to exist. Otherwise, the solution is forced to adjust at the boundary in a non-smooth manner to satisfy the boundary condition \eqref{eq:Robin-type-BC} instantaneously for $t>0$. From a numerical perspective, it is also preferable to choose initial data that are already compatible with the boundary condition, as this typically enhances numerical stability.
\end{remark}

We now state the following result on the solution of the initial boundary value problem~\eqref{eq:mainPDE.bis} with initial condition $f^0$.

\begin{lemma} \label{lemma:g.P2-advection}
Assume that $f$ is a weak solution to the BDY PDE problem \eqref{eq:mainPDE.bis}, starting from an initial condition $f^0 \in \mathcal{P}_1(\mathbb{R}_+)$ which is a smooth probability density with unitary mass and mean value $\mu > 0$. Then
\begin{equation}\label{eq:conservation_laws}
\frac{\dd}{\dd t} \int_0^\infty f(v,t)\,\dd v = 0 \quad \textrm{and} \quad \frac{\dd}{\dd t} \int_0^\infty v\,f(v,t)\,\dd v = 0.
\end{equation}
In other words, the BDY PDE preserves both the total probability mass and the mean value, and $f(\cdot,t) \in L^1(\R_+)$ for $t \ge 0$ with $||f(\cdot,t)||_{L^1} \equiv 1$. Moreover, the Boltzmann-Gibbs distribution
$f^\infty$ \eqref{eq:BG} is the unique equilibrium solution of \eqref{eq:mainPDE.bis}. Furthermore, any weak solution $f$ to~\eqref{eq:mainPDE.bis} issuing from an initial condition $f^0\in\cP_2(\R_+)$ is such that $f\in\cP_2(\R_+)$ for all $t>0$, i.e. 
 there exists $\cM_2>0$ which is a uniform bound for the second order moment of $f$.
\end{lemma}
\begin{proof}
Setting $\varphi(v) \equiv 1$ and $\varphi(v) = v$ respectively in~\eqref{eq:FP-weak.2} the weak formulation, we obtain the conservation of total probability mass and the mean value. Alternatively, a straightforward computation gives rise to (choosing $\lambda=2$ for brevity)
\begin{equation}\label{eq:mass_conservation}
\begin{aligned}
\frac{\dd}{\dd t} \int_0^\infty f(v,t)\,\dd v &= \int_0^\infty \partial_{vv} f(v,t)\,\dd v + f(0,t)\,\int_0^\infty \partial_v f(v,t)\,\dd v \\
&= -\partial_v f(0,t) - f^2(0,t) = 0,
\end{aligned}
\end{equation}
where the last equality follows from the nonlinear Robin-type boundary condition. On the other hand, we perform two integrations by parts to obtain
\begin{equation}\label{eq:mean_conservation}
\begin{aligned}
\frac{\dd}{\dd t} \int_0^\infty v\,f(v,t)\,\dd v &= \int_0^\infty v\,\partial_{vv} f(v,t)\,\dd v + f(0,t)\,\int_0^\infty v\,\partial_v f(v,t)\,\dd v \\
&= -\int_0^\infty \partial_v f(v,t)\,\dd v - f(0,t)\,\int_0^\infty f(v,t)\,\dd x \\
&= f(0,t) - f(0,t) = 0.
\end{aligned}
\end{equation}
In addition, we can readily verify that the Boltzmann-Gibbs distribution $f^\infty$~\eqref{eq:BG} constitutes the unique equilibrium solution of~\eqref{eq:mainPDE.bis}. We also recall that the Boltzmann-Gibbs distribution admits finite momenta of order two which is equal to $2\, \mu^2$. 
Then in the same spirit as the proof of Lemma~\ref{lem:epsilon_BDY}, if we set $\varphi(v)=v^2$ 
in~\eqref{eq:FP-weak.2}, we deduce that
\[\dfrac{\dd}{\dd t} M_2(t)= 2\,(1-\mu\,f(0,t)).
\]
In general, $M_2$ 
is not conserved and we cannot state anything about its monotonicity. However, even in the worst case scenario that they were monotonically increasing, we know that their stationary asymptotic solutions are finite. Therefore we can conclude that if the initial condition $f^0\in \cP_2(\R_+)$ and $f(0,t)$ is bounded in time, then the second moment of $f$ is finite for all $t \geq 0$ and bounded by a positive constant $\cM_2>0$, whence $f\in\cP_2(\R_+)$.
\end{proof}

As already mentioned, the BDY PDE has been rigorously studied in~\cite{calvez_one_2010,calvez_analysis_2012,lepoutre_cell_2014}.
In~\cite{calvez_analysis_2012} the authors prove that the BDY PDE admits a global-in-time weak solution. In particular, under simple assumptions on the initial condition, which is required to satisfy
\[
f^0 \in \cP_1(\R_+), \qquad \int_{\R_+}f^0(v) \, (\log f^0)_+ \, \dd v < \infty
\]
the BDY-PDE admits a global weak solution in the following sense 
\begin{equation}\label{hp.f0}
    f \in L^\infty((0,T),L^1(\R_+)), \qquad \partial_v f \in L^1((0,T) \times \R_+)
\end{equation}
for all $T>0$ and if $f$ is a solution of the BDY-PDE in the sense of distributions $\mathcal{D}'(\R_+)$, i.e. of the weak form of the BDY-PDE with test functions in $\mathcal{D}(\R_+) =\lbrace \varphi \in \mathcal{C}^\infty(\R_+)| \, {\rm supp}(\varphi)=K \subset \R_+, \, K {\rm compact}\rbrace$.  Our derivation and existence result is given for $\mathcal{B}^3$, but is could be extended by restricting the space of test functions by choosing $m = \infty$, and recalling that $\mathcal{D} \subset \mathcal{B}^\infty$. Moreover, the solution satisfies
\begin{equation}\label{thm.Calvez}
\sup_{t \in (0,T)} \int_{\R_+} f(v,t) (\log f(v,t))_+ \, \dd v < +\infty, \qquad \int_0^T \int_{\R_+} f(v,t) (\partial_v \log f(v,t))^2 \, \dd v \, \dd t < +\infty.
\end{equation}
Moreover, as $f(0,t)= -\int_{\R_+}\partial_v f(v,t) \, \dd v$, then
\begin{equation}
    f^2(0,t) \le \int_{\R_+}(\partial_v \log f(v,t))^2 \, f(v,t) \, \dd v.
\end{equation}
As a consequence, thanks to~\eqref{thm.Calvez}, $f(0,t) \in L^2((0,T))$.


\section{Convergence to Boltzmann-Gibbs distribution}\label{sec:sec3}

\subsection{Entropy dissipation}\label{subsec:sec3.1}

Our main objective in this section is to prove that solutions $f(v,t)$ of the BDY PDE \eqref{eq:mainPDE} converge, in the large time limit $t \to \infty$, to its Boltzmann-Gibbs equilibrium distribution $f^\infty$ \eqref{eq:BG}. To start with, we show that the relative entropy
\begin{equation}\label{def:Shannon}
\cH\left[f \mid f^\infty\right](t) \coloneqq \int_0^\infty f(v,t)\,\ln \frac{f(v,t)}{f^\infty(v)}\,\dd v.
\end{equation}
serves as a Lyapunov functional for the evolution equation \eqref{eq:mainPDE}, which decreases monotonically in time.

\begin{proposition}\label{prop:entropy_dissipation}
Let $f$ be a solution to~\eqref{eq:mainPDE}. Under the settings of Lemma \ref{lemma:g.P2-advection}, for all $t\geq 0$ it holds that
\begin{equation}\label{eq:entropy_dissipation}
\frac{\dd}{\dd t} \cH\left[f \mid f^\infty\right] = -\cD[f] \leq 0,
\end{equation}
where \begin{equation}\label{eq:dissipation_entropy}
\cD[f(\cdot,t)] \coloneqq \int_0^\infty \frac{\left(\partial_v f(v,t) + f(0,t)\,f(v,t)\right)^2}{f(v,t)}\,\dd v.
\end{equation}
\end{proposition}

\begin{proof}
Thanks to the conservation of the total probability mass and the mean value \eqref{eq:conservation_laws}, we deduce that
\begin{equation*}
\begin{aligned}
&\dfrac{d}{dt}\cH\left[f \mid f^\infty\right](t) = \frac{\dd}{\dd t} \int_0^\infty f(v,t)\,\ln f(v,t) \,\dd v = \int_0^\infty \partial_t f(v,t)\,\ln f(v,t) \,\dd v \\
&= \int_0^\infty \partial_{vv} f(v,t)\,\ln f(v,t)\,\dd v + f(0,t)\,\int_0^\infty \partial_v f(v,t)\,\ln f(v,t)\,\dd v \\
&= -\partial_v f(0,t)\,\ln f(0,t) - \int_0^\infty \frac{|\partial_v f(v,t)|^2}{f(v,t)}\,\dd v - f(0,t)^2\,\ln f(0,t) - f(0,t)\,\int_0^\infty \partial_v f(v,t)\,\dd v \\
&= f^2(0,t) - \int_0^\infty \frac{|\partial_v f(v,t)|^2}{f(v,t)}\,\dd v - \ln f(0,t)\cdot\underbrace{\left(\partial_v f(0,t) + f^2(0,t)\right)}_{=~0~\textrm{by boundary condition}} \\
&= f^2(0,t) - \int_0^\infty \frac{|\partial_v f(v,t)|^2}{f(v,t)}\,\dd v = -\int_0^\infty \frac{\left(\partial_v f(v,t) + f(0,t)f(v,t)\right)^2}{f(v,t)} \, \dd v,
\end{aligned}
\end{equation*}
whence the proof is completed.
\end{proof}
We include here a numerical experiment illustrating the entropic decay of $f$ toward the Boltzmann-Gibbs distribution
$f^\infty$. Figure~\ref{fig:entropic_decay}-left displays the time evolution of the solution $f$ to the BDY PDE \eqref{eq:mainPDE},
initialized with the Gamma-type distribution $f(v,0) \coloneqq \frac{2}{\mu}\,\left(1-\frac{v}{2\,\mu}\right)^2\,\expo^{-v/\mu}$ with $\mu = 1$, which is consistent with the Robin-type boundary condition \eqref{eq:Robin-type-BC}. The corresponding decay of the relative entropy
$\cH\left[f \mid f^\infty\right]$ is shown in Figure~\ref{fig:entropic_decay}-right.

\begin{figure}[!htb]
  \begin{subfigure}{0.47\textwidth}
    \centering
    \includegraphics[scale=0.6]{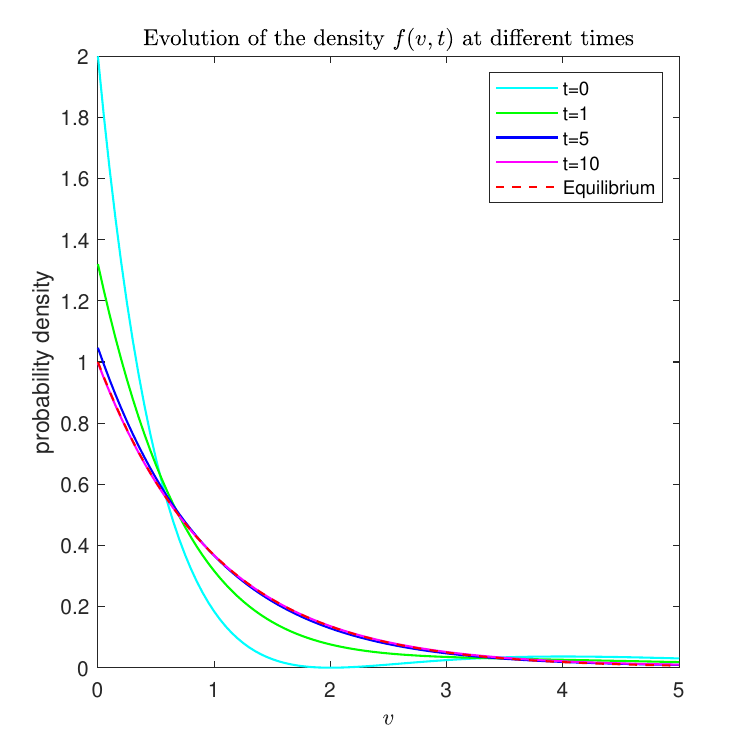}
  \end{subfigure}
  \hspace{0.1in}
  \begin{subfigure}{0.47\textwidth}
    \centering
    \includegraphics[scale=0.6]{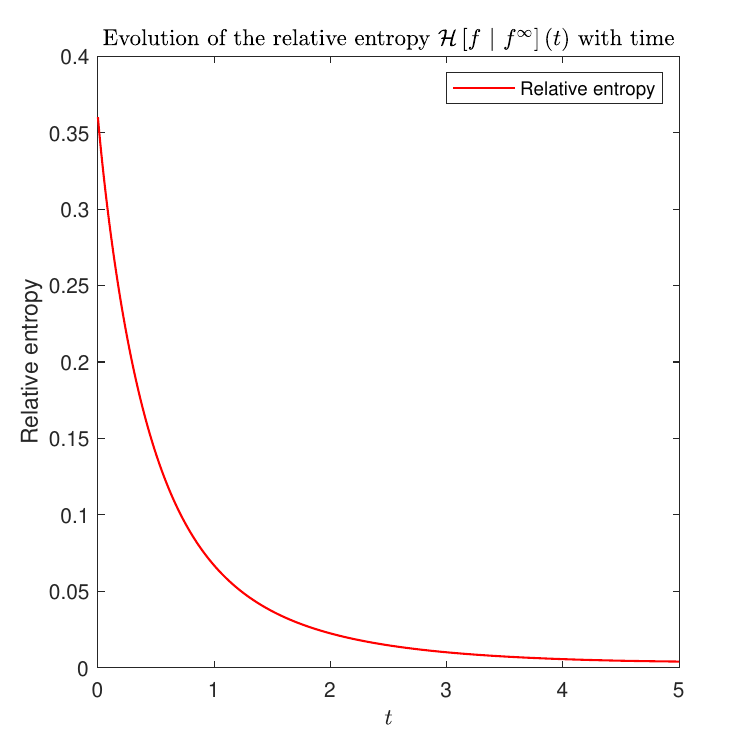}
  \end{subfigure}
  \caption{{\bf Left}: Simulation of the Bennati–Dragulescu–Yakovenko PDE problem \eqref{eq:mainPDE}. {\bf Right}: Evolution of the relative entropy $\cH\left[f \mid f^\infty\right](t)$.}
  \label{fig:entropic_decay}
\end{figure}

The computations in the proof of Proposition~\ref{prop:entropy_dissipation} can be generalized and carried out in a more systematic manner. Indeed, in the same spirit as~\cite{furioli_2017, torregrossa_wealth_2017}, we first reformulate the Fokker-Planck equation. For this purpose, we refer to its flux form~\eqref{eq:flux} with $\lambda=2$. If we consider the relation which defines the stationary state $f^\infty$, i.e., $\mathcal{F}[f^\infty]=0$, we obtain
\[\partial_v f^\infty(v) +f^\infty(0)f^\infty(v)= f^\infty(v) \left(\partial_v \ln f^\infty(v) +f^\infty(0)\right)=0,\]
which implies
\begin{equation}
f^\infty(0)=-\dfrac{\partial_v f^\infty(v)}{f^\infty(v)}=-\partial_v \ln f^\infty(v), \quad \forall v \in \R_+.
\end{equation}
In particular, we have that $f^\infty(0)=1/\mu$. Then we rewrite the flux $\mathcal{F}[f]$ as
\[
\begin{aligned}
\mathcal{F}[f](v,t)=& f(v,t) \left( \partial_v \ln f(v,t) + f(0,t) \right)\\
=& f(v,t) \left(\partial_v \ln f(v,t) -\partial_v\ln f^\infty(v)+f(0,t) - f^\infty(0)\right).
\end{aligned}
\]
We remark that the boundary conditions satisfied by $f^\infty$ and $f$ yield
\begin{equation}\label{eq:cb.2}
f^\infty(0)=- \dfrac{\partial_v f^\infty(0)}{f^\infty(0)} \quad \textrm{and} \quad f(0,t) = -\dfrac{\partial_v f(0,t)}{f(0,t)},
\end{equation}
leading us to
\[
f(0,t)-f^\infty(0)= -\dfrac{\partial_v f(0,t)}{f(0,t)}+\dfrac{\partial_v f^\infty(0)}{f^\infty(0)} = - \partial_v \ln \dfrac{f(v,t)}{f^\infty(v)}\Big|_{v=0}.
\]
As a consequence, an equivalent formulation of the problem~\eqref{eq:mainPDE} is
\begin{equation}\label{eq:FP2}
\partial_t f(v,t) = \partial_v \left[ f(v,t) \left( \partial_v \ln \dfrac{f(v,t)}{f^\infty(v)}-  \partial_v \ln \dfrac{f(v,t)}{f^\infty(v)}\Big|_{v=0}\right)\right].
\end{equation}
\begin{remark}
We observe that the boundary condition $\mathcal{F}[f](0,t)=0$ is automatically satisfied and does not actually need to be imposed. This is a consequence of the fact that the desired boundary condition is utilized in order to reformulate the Fokker-Planck equation~\eqref{eq:FP2}.
\end{remark}

Next, we define the ratio
\begin{equation}\label{def:F}
F(v,t) \coloneqq \dfrac{f(v,t)}{f^\infty(v)}.
\end{equation}
Inserting \eqref{def:F} into \eqref{eq:FP2} gives rise to
\begin{equation}\label{eq:F}
\begin{cases}
f^\infty(v)\,\partial_t F(v,t) = \partial_v \left[ f^\infty(v) \left(\partial_v F(v,t) -F(v,t) \Lambda(t)\right) \right], &~~v > 0,~t\geq 0\\
\partial_v F(v,t) = F(v,t)\,\Lambda(t), &~~v = 0,~t\geq 0,
\end{cases}
\end{equation}
where $\Lambda(t) \coloneqq \dfrac{\partial_v F(0,t)}{F(0,t)}$. Again, the boundary condition does not really need to be imposed as it is automatically satisfied at $v=0$ by the definition of $\Lambda$.  Therefore, the evolution of $F$ is governed by
\begin{equation}\label{eq:rF}
\partial_t F(v,t) = \partial_{vv} F(v,t) -\partial_v F(v,t)\left(\Lambda(t)+f^\infty(0)\right)+f^\infty(0)\,\Lambda(t)\, F(v,t),
\end{equation}
where we emphasize that the quantity $\Lambda=\Lambda(t)$ embodies the boundary condition. A routine computation shows that $\Lambda(t)$ can be rewritten as
\begin{equation}\label{def:Lambda}
\Lambda(t) = f^\infty(0) - f(0,t).
\end{equation}

\begin{proposition}\label{teo.1} 
Assume that $F$ is a solution of~\eqref{eq:rF}. Then if $\Psi \colon \R_+\to \R$ is a smooth function such that
\[\sup\limits_{v \in \R_+} |\Psi(v)| \leq c < \infty,\]
then the following relation holds:
\[\int_{\R_+} f^\infty(v)\,\Psi(v)\,\partial_t F(v,t)\, {\rm d}v = -\int_{\R_+} f^\infty(v)\, \partial_v \Psi(v)\left(\partial_v F(v,t)-\Lambda(t)\, F(v,t)\right) \, {\rm d}v.\]
\end{proposition}
\begin{proof}
We resort to the equivalent formulation~\eqref{eq:F} and perform integration by parts as follows:
\begin{align*}
\int_{\R_+} f^\infty(v)\,\Psi(v)\,\partial_t F(v,t)\, {\rm d}v &= \left[\Psi(v)\,f^\infty(v)\left(\partial_v F(v,t)-\Lambda(t)\, F(v,t)\right)\right]\vert_{v=0}\\
&\phantom{=}-\int_{\R_+}\partial_v \Psi(v)\, f^\infty(v) \left(\partial_v F(v,t)-\Lambda(t)\,F(v,t)\right)\,\dd v.
\end{align*}
Thus the advertised identity follows as the boundary term vanishes.
\end{proof}

\noindent We can now prove the following:
\begin{theorem}\label{Teo1} 
Assume that $\Phi \colon \R_+\to \R$ is a smooth and convex function. If $F(v,t)$ is the solution to~\eqref{eq:rF} and $c \leq F(v,t)\leq C$ for some positive constants $c < C$ and all $t\geq 0$, then the generalized entropy functional
\begin{equation}\label{def:Theta}
\Theta[F](t) \coloneqq \int_{\R_+} f^\infty(v)\,\Phi(F(v,t)) \, {\rm d}v
\end{equation}
has a time variation given by the following identity
\[\frac{d}{dt}\Theta[F](t) = -I_{\Theta}[F](t),\]
where $I_{\Theta}[F](t)$ is defined via
\begin{equation}\label{def:entropy.prod}
\begin{aligned}
I_{\Theta}[F](t) \coloneqq &\int_{\R_+} f^\infty(v)\, \Phi''(F(v,t))\,|\partial_v F(v,t)|^2 \, {\rm d} v\\
&-\Lambda(t) \int_{\R_+} f^\infty(v)\, \Phi''(F(v,t))\, F(v,t)\, \partial_v F(v,t) \, {\rm d} v,
\end{aligned}
\end{equation}
which can be viewed as the generalized entropy production functional.
\end{theorem}
\begin{proof}
The relation~\eqref{def:entropy.prod} follows directly from Proposition~\ref{teo.1} by choosing $\Psi(v)=\Phi'(F(v,t))$ for a fixed $t\geq 0$.
\end{proof}
\begin{remark}
The hypothesis of a uniform bound on $F$ is of course a strong assumption. It is due to the fact that a general result concerning the maximum principle for a general nonlinear Fokker-Planck equation is still missing, and~\eqref{eq:mainPDE} does not fall in the rather general case treated by Le Bris and Lions in~\cite{lebris_lions_2008}. The presence of $f(0,t)$ in the drift term does not allow to conclude by considering a maximum principle for uniformly parabolic differential equations. We emphasize here that this issue also underlies the hypothesis of uniform bounds on both $f_\varepsilon, f$ and their derivatives in a neighborhood of $v=0$ in the previous section.
\end{remark}
Differently with respect to~\cite{furioli_2017}, the quantity $I_{\Theta}$ is not identically positive. In fact, the first term, which also appears in~\cite{furioli_2017}, is positive due to the convexity assumption on $\Phi$. However, the second term is a new contribution and does not always have a fixed sign. This indefiniteness is mainly due to the sign of $\Lambda$. If $\Lambda \equiv 0$ then we would go back to the original case, as $f(0,t)$ in equation~\eqref{eq:mainPDE} would not be time dependent. This is the key difference with respect to~\cite{furioli_2017}, which arises directly in the derivation of the evolution equation for $F$. In fact, the Fokker–Planck equation investigated in \cite{furioli_2017} is linear with constant-in-time diffusion and drift coefficients, and it contains no nonlinear term involving $f(0,t)$. As a consequence, the authors in~\cite{furioli_2017} do not need to exploit the boundary conditions~\eqref{eq:cb.2} in the reformulation of the flux term $\mathcal{F}[f]$ in order to write the equivalent Fokker-Planck equation~\eqref{eq:FP2}. As already mentioned, this also underlies the fact that boundary condition associated with \eqref{eq:FP2} is automatically satisfied.

In summary, we cannot conclude by Theorem \ref{Teo1} that the functional $\Theta$ is monotonically decreasing in time. As a consequence, we now analyze different kinds of relative entropy functionals.
First, we consider the relative entropy~\eqref{def:Shannon}, which is obtained in the aforementioned framework by setting $\Phi(F)=F\ln F$ in~\eqref{def:Theta}. As a consequence,
\[
I_{\Theta}[F](t) = \int_{\R_+} f^\infty(v)\, \dfrac{1}{F(v,t)}\,|\partial_v F(v,t)|^2 \, \dd v - \Lambda(t) \int_{\R_+} f^\infty(v)\,\partial_v F(v,t) \, \dd v.
\]
The first term is obviously non-negative, and it can also be rewritten as 
\[
\int_{\R_+} f^\infty(v)\, \dfrac{1}{F(v,t)}\,|\partial_v F|^2 \, \dd v = \int_{\R_+} f(v,t) \left( \dfrac{\partial_v f(v,t)}{f(v,t)}-\dfrac{\partial_v f^\infty(v)}{f^\infty(v)}\right)^2 \, \dd v,
\]
in which the right-hand side is the Fisher information of $f$ relative to $f^\infty$. On the other hand, the second term can be rewritten as
\begin{align*}
\Lambda(t) \int_{\R_+} f^\infty(v)\,\partial_v F(v,t) \, \dd v &= \Lambda(t) \left[ (f^\infty(v)\,F(v,t))\vert_{v=0}^\infty -\int_{\R_+} \partial_v f^\infty(v)\, F(v,t) \, \dd v \right]\\
& = \Lambda(t) \left[ -f(0,t)-\int_{\R_+} \dfrac{\partial_v f^\infty(v)}{f^\infty(v)}\,f(v,t) \, \dd v\right] \\
&= \Lambda(t) \left[-f(0,t)+f^\infty(0) \right] = (f(0,t)-f^\infty(0))^2,
\end{align*}
where in the second equality we have employed the boundary condition for $f^\infty(0)$ and the fact that $f$ has a constant unitary mass.
As a consequence, we can show that $I_{\Theta}[F] = \cD[f(\cdot,t)]$ is non-negative.
Moreover, by virtue of \eqref{eq:cb.2}, we have
\[
(f(0,t)-f^\infty(0))^2 = \left(\dfrac{\partial_v f^\infty(0)}{f^\infty(0)}-\dfrac{\partial_v f(0,t)}{f(0,t)}\right)^2 = \int_{\R_+} f(v,t) \left(\dfrac{\partial_v f^\infty(0)}{f^\infty(0)}-\dfrac{\partial_v f(0,t)}{f(0,t)}\right)^2 \, \dd v.
\]

\noindent In conclusion, we end up with
\begin{equation*}
I_{\Theta}[F](t) = \displaystyle\int_{\R_+}f(v,t) \left[\left( \dfrac{\partial_v f(v,t)}{f(v,t)}-\dfrac{\partial_v f^\infty(v)}{f^\infty(v)}\right)^2-\left( \dfrac{\partial_v f(0,t)}{f(0,t)}-\dfrac{\partial_v f^\infty(0)}{f^\infty(0)}\right)^2\right] \, \dd v.
\end{equation*}
If $\Lambda \equiv 0$, the entropy production induced by the relative entropy boils down to the Fisher information of $f$ relative to $f^\infty$, which is defined for two smooth probability densities $f_1$ and $f_2$ by
\begin{equation}\label{def:Fisher}
I(f_1,f_2) \coloneqq \int_{\R_+}f_1(v) \left( \dfrac{\partial_v f_1(v)}{f_1(v)}-\dfrac{\partial_v f_2(v)}{f_2(v)}\right)^2 \, \dd v.
\end{equation}
In general, we introduce the following generalized relative Fisher information:
\begin{definition}
Let $f_1$ and $f_2$ be two smooth probability densities, and let $\Lambda$ be defined by~\eqref{def:Lambda} (with $f$ and $f^\infty$ replaced by $f_1$ and $f_2$, respectively). Then, the Fisher information of $f_1$ relative to $f_2$ with boundary term $\Lambda$ is defined by
\begin{equation}\label{def:Fisher.Lambda}
I_\Lambda(f_1,f_2) \coloneqq \displaystyle\int_{\R_+}f_1(v) \left[ \dfrac{\partial_v f_1(v)}{f_1(v)}-\dfrac{\partial_v f_2(v)}{f_2(v)} - (f_2(0) - f_1(0))\right]^2\, \dd v
\end{equation}
\end{definition}
As a consequence, we have
\begin{equation}\label{eq:Shannon.Fisher}
\dfrac{\dd}{\dd t} \cH\left[f \mid f^\infty\right](t) = -I_{\Lambda(t)}\left(f(\cdot,t),f^\infty\right),
\end{equation}
where
\begin{equation}\label{eq:Shannon.Lambda}
I_{\Lambda(t)}\left(f(\cdot,t),f^\infty\right) = I\left(f(\cdot,t),f^\infty\right) - \Lambda^2(t).
\end{equation}

\begin{remark}
    As observed in~\cite{lepoutre_cell_2014}, computations show that the entropy dissipation is the relative Fisher information with respect to the instantaneous exponential density $\nu(v,t) \coloneqq f(0,t)\,\expo^{-f(0,t)\,v}$. That is, 
    \[
    I_{\Lambda(t)} (f(\cdot,t),f^\infty)= I(f(\cdot,t),\nu(\cdot,t)).
    \]
\end{remark}

To conclude, we prove an inequality which relates the Fisher information $I\left(f(\cdot,t),f^\infty\right)$ to the Hellinger distance, analogous to the one which was first proved by Johnson and Barron \cite{johnson_fisher_2004} with $f^\infty$ equal to the Maxwellian (see also \cite{furioli_2017} for relevant discussions). To this aim, let us now introduce the Hellinger distance.
\begin{definition}
For any pair of nonnegative functions $f_1$ and $f_2$ defined on $\R_+$, the Hellinger distance is defined by
\begin{equation}
\dd_H(f_1,f_2) \coloneqq \left( \int_{\R_+} \left(\sqrt{f_1(v)}-\sqrt{f_2(v)}\right)^2 \, \dd v\right)^{1/2}.
\end{equation}
\end{definition}

\begin{proposition}\label{prop:Hellinger_Fisher}
Assume that $f$ is a solution to the BDY PDE \eqref{eq:mainPDE}, issuing from a smooth initial condition $f(\cdot,0)$. Then
\begin{equation}
d_H^2(f,f^\infty) \le 2\, \mu^2\,I(f(\cdot,t),f^\infty),
\end{equation}
where $f^\infty$ is the Boltzmann-Gibbs distribution~\eqref{eq:BG},  and $\mu$ is the conserved average.
\end{proposition}
\begin{proof}
We recall the Poincar\'e inequality satisfied by the exponential probability measure $f^\infty$, which reads as 
\begin{equation}\label{eq:poincare_expo}
\operatorname{Var}_{f^\infty}(g) \coloneqq \int_0^\infty g^2(v)\,f^\infty(v)\,\dd v - \left( \int_0^\infty g(v)\,f^\infty(v)\,\dd v\right)^2 \leq 4\,\mu^2\, \int_0^\infty |g'(v)|^2\,f^\infty(v)\,\dd v.
\end{equation}
Applying the Poincar\'e inequality \eqref{eq:poincare_expo} 
with \[g(v,t) \coloneqq \sqrt{\frac{f(v,t)}{f^\infty(v)}} \quad \text{and} \quad \partial_v g(v,t) = \frac 12\,g(v,t)\,\partial_v \ln\left(\frac{f(v,t)}{f^\infty(v)}\right),\] we obtain 
\[0\leq 1 - \left(\int_0^\infty \sqrt{f(v,t)\,f^\infty(v)}\,\dd v\right)^2 \leq \mu^2\,\int_0^\infty f(v,t)\,\left|\partial_v \ln\left(\frac{f(v,t)}{f^\infty(v)}\right)\right|^2\,\dd v = \mu^2\,I\left(f(\cdot,t),f^\infty\right).\] As a result, we deduce that 
\begin{align*}
\int_0^\infty \left(\sqrt{f(v,t)} - \sqrt{f^\infty(v)}\right)^2\,\dd v &= 2\,\left(1-\int_0^\infty \sqrt{f(v,t)\,f^\infty(v)}\,\dd v\right) \leq 2\,\left(1-\left(\int_0^\infty \sqrt{f(v,t)\,f^\infty(v)}\,\dd v\right)^2\right) \\
&\leq 2\,\mu^2\,I\left(f(\cdot,t),f^\infty\right).
\end{align*}
\end{proof}
An immediate consequence is the following
\begin{proposition}\label{prop:Lambda.D}
Assume that $f$ is a solution to the BDY PDE \eqref{eq:mainPDE}, issuing from a smooth initial condition $f^0 \in \cP_2(\R_+)$. Then
\begin{equation}
\mu^2\,\Lambda^2(t) \le \cM_2 \,\cD[f(\cdot,t)], \qquad I(f(\cdot,t),f^\infty) \le \left(1 + \frac{\cM_2}{\mu^2}\right)\,\cD[f(\cdot,t)],
\end{equation}
where $\cM_2>0$ is the uniform bound on the second moment. 
\end{proposition}
Thanks to the identity
\[\mu\,(f(0,t) - f^\infty(0)) = \int_0^\infty v\,\left(\partial_v f(v,t) + f(0,t)\,f(v,t)\right)\,\dd v, \]
the Cauchy-Schwarz inequality leads us to 
\[\mu^2\,\left(f(0,t) - f^\infty(0)\right)^2 \leq \int_0^\infty v^2\,f(v,t)\,\dd v\,\int_0^\infty \frac{\left(\partial_v f(v,t) + f(0,t)\,f(v,t)\right)^2}{f(v,t)}\,\dd v \leq \cM_2\,\cD[f(\cdot,t)],\]
where the uniform bound $\cM_2$ is granted by $f^0 \in \cP_2(\R_+)$, and by Lemma~\ref{lemma:g.P2-advection}.
Consequently, as $I_\Lambda =\cD$, and in view of~\eqref{def:Fisher.Lambda}, we have 
\begin{equation}\label{eq:I_to_D}
I\left(f(\cdot,t),f^\infty\right) = \cD[f(\cdot,t)] + \left(f(0,t) - f^\infty(0)\right)^2 \leq \left(1 + \frac{\cM_2}{\mu^2}\right)\,\cD[f(\cdot,t)].
\end{equation}.


\subsection{Quantitative convergence in relative entropy}\label{subsec:sec3.2}

In this section, we aim to establish a quantitative large-time convergence result for the solution $f(\cdot,t)$ to the BDY PDE \eqref{eq:mainPDE} towards its Boltzmann-Gibbs equilibrium $f^\infty(\cdot)$, in the sense of relative entropy. Although such quantitative convergence guarantees have been established in a few earlier works related to cell polarisation \cite{calvez_one_2010,calvez_analysis_2012,lepoutre_cell_2014}, based on a suitable modified logarithmic Sobolev inequality satisfied by the exponential law, for the sake of completeness, we present a different argument inspired from a recent work on the classical BDY dynamics on $\mathbb N$ \cite{cao_interacting_2024}. Indeed, we establish the following quantitative convergence result in relative entropy:
\begin{theorem}\label{thm:quantitative_entropy_dissipation}
Assume that $f$ is a solution to the BDY PDE \eqref{eq:mainPDE}. In addition, we assume that 
\begin{equation}\label{eq:assumption_expo_bound}
\sup_{t\geq 0, v\geq 0} \expo^{\alpha\,v}\,f(v,t) < \infty
\end{equation}
for some $\alpha > 0$. Then there exist fixed constants $c, C > 0$ such that 
\begin{equation}\label{eq:quantitative_entropy_dissipation}
 \int_0^\infty f(v,t)\,\ln \frac{f(v,t)}{f^\infty(v)}\,\dd v \leq C\,\expo^{-c\sqrt{t}} \quad \text{for all $t \geq 0.$}
\end{equation}
\end{theorem}

\begin{proof}
Throughout the proof, we use $C > 0$ to denote a generic positive constant which is independent of time, and whose value might change from line to line. For $M \geq \expo > 0$ to be chosen later (which will depend on the time $t\geq 0$), we split the relative entropy $\cH\left[f \mid f^\infty\right](t)$ into two parts as follows:
\[\int_0^\infty f(v,t)\,\ln \frac{f(v,t)}{f^\infty(v)}\,\dd v = \int_0^\infty \left[f(v,t)\,\ln \frac{f(v,t)}{f^\infty(v)} + f^\infty(v) - f(v,t)\right]\,\dd v = \cH_1(t) + \cH_2(t),\]
where \[\cH_1(t) \coloneqq \int_{v: f(v,t) \geq M\,f^\infty(v)} \left[f(v,t)\,\ln \frac{f(v,t)}{f^\infty(v)} + f^\infty(v) - f(v,t)\right]\,\dd v\] and \[\cH_2(t) \coloneqq \int_{v: f(v,t) < M\,f^\infty(v)} \left[f(v,t)\,\ln \frac{f(v,t)}{f^\infty(v)} + f^\infty(v) - f(v,t)\right]\,\dd v.\]
Thanks to assumed bound \eqref{eq:assumption_expo_bound}, we have 
\begin{align*}
\cH_1(t) &\leq C\,\int_{v: f(v,t) \geq M\,f^\infty(v)} f(v,t)\,(v+C)\,\dd v \leq \frac{C}{M^\theta}\,\int_{v: f(v,t) \geq M\,f^\infty(v)} \expo^{\alpha\,v}\,f(v,t)\,\left(\frac{f(v,t)}{f^\infty(v)}\right)^\theta\,\frac{v+C}{\expo^{\alpha\,v}}\,\dd v\\
&\leq \frac{C}{M^\theta}\,\sqrt{\int_0^\infty \expo^{2\,\alpha\,v}\,f(v,t)^{2+2\,\theta}\,\dd v}\,\sqrt{\int_0^\infty \frac{(v+C)^2}{(f^\infty(v))^{2\,\theta}\,\expo^{2\,\alpha\,v}}\,\dd v} \leq \frac{C}{M^\theta},
\end{align*}
where $\theta > 0$ is chosen such that $\theta < \alpha\,\mu$. To control $\cH_2(t)$, we notice that for any $a,b > 0$, if $a \leq M\,b$ for some $M \geq \expo$, then 
\begin{equation}\label{eq:elementary}
\frac{a}{b}\,\ln\frac{a}{b} + 1 - \frac{a}{b} \leq C\,(1 + \ln M)\,\left(\sqrt{\frac{a}{b}}-1\right)^2
\end{equation}
for some universal constant $C > 0$ independent of $a$, $b$, and $M$. Indeed, to show \eqref{eq:elementary}, it suffices to notice that the function $r \mapsto \frac{r\,\ln r + 1-r}{(\sqrt{r} - 1)^2}$ extends continuously at $r = 1$ (with limit $2$) and is bounded on $r \in [0,4]$. On the other hand, for $r \geq 4$, we have $(\sqrt{r} - 1)^2 \geq r/4$ and $r\,\ln r + 1-r \leq r\,\ln r$, which finishes the proof of \eqref{eq:elementary}. Therefore, we deduce that
\[\cH_2(t) \leq C\,(1 + \ln M)\,\int_0^\infty \left(\sqrt{f(v,t)} - \sqrt{f^\infty(v)}\right)^2\,\dd v = C\,(1 + \ln M)\, d_H^2(f,f^\infty).\]
Considering Propositions~\ref{prop:Hellinger_Fisher}-\ref{prop:Lambda.D}, which hold true thanks to~\eqref{eq:assumption_expo_bound} which grants the uniform bound on the second moment, and putting these estimates together, we obtain 
\[\cH\left[f \mid f^\infty\right](t) \leq \frac{C}{M^\theta} + C\,(1 + \ln M)\,\cD[f(\cdot,t)].\]
Let $\cH_* \coloneqq \max\{1, \cH\left[f \mid f^\infty\right](0), 2\,C\}$ and take 
\[M \coloneqq \max\left\{\expo,\left(\frac{2\,C}{\cH\left[f \mid f^\infty\right](t)}\right)^{1/\theta}\right\}.\]
Then $C\,M^{-\theta} \leq \frac{\cH\left[f \mid f^\infty\right](t)}{2}$, whence 
\[\frac{\cH\left[f \mid f^\infty\right](t)}{2} \leq C\,(1 + \ln M)\,\cD[f(\cdot,t)] \leq C\,\left(1 + \ln \left(\frac{\cH_*}{\cH\left[f \mid f^\infty\right](t)}\right)\right)\,\cD[f(\cdot,t)].\]
Consequently, we arrive at the following differential inequality:
\[\frac{\dd}{\dd t} \cH\left[f \mid f^\infty\right](t) = -\cD[f(\cdot,t)] \leq -C\,\frac{\cH\left[f \mid f^\infty\right](t)}{1+\ln\left(\frac{\cH_*}{\cH\left[f \mid f^\infty\right](t)}\right)}, \] from which the claimed bound \eqref{eq:quantitative_entropy_dissipation} follows.
\end{proof}

\begin{remark}
Quantitative convergence rates of order $t^{-1/2}$ were established in \cite{lepoutre_cell_2014} assuming finite initial entropy and a finite second moment, and improved to $t^{-1}$ under a uniform third-moment bound. Under our stronger uniform exponential-tail assumption \eqref{eq:assumption_expo_bound}, an entropy-splitting argument combined with the Poincar\'e inequality for exponential measures yields the sharper rate displayed in \eqref{eq:quantitative_entropy_dissipation}. This estimate complements those algebraic bounds by providing accelerated convergence within a more restricted class of solutions. 
\end{remark}

\subsection{Convergence in $L^1$, Hellinger and Fourier distance}\label{subsec:sec3.3}

From~\eqref{eq:Shannon.Fisher}, we obtain \begin{equation}\label{eq:H.I1}
\int_0^t \dfrac{\dd}{\dd \tau} \cH\left[f \mid f^\infty\right](\tau)\,\dd \tau = -\int_0^t \mathcal{D}[f(\cdot,\tau)] \, \dd \tau.
\end{equation}
Given the initial condition $f^0$, it follows that
\[
\cH\left[f \mid f^\infty\right](t)-\cH\left[f^0 \mid f^\infty\right] = - \int_0^t \mathcal{D}[f(\cdot,\tau)] \, \dd \tau
\]
whence for all $t\geq 0$ it holds that
\begin{equation}\label{eq:I.L1}
 \int_0^t \mathcal{D}[f(\cdot,\tau)] \, \dd \tau\le \cH\left[f^0 \mid f^\infty\right].
\end{equation}
As a result, the entropy production $\cD[f(\cdot,t)]$ is an $L^1(\R_+)$ function of time if the initial entropy $\cH\left[f^0 \mid f^\infty\right]$ is finite. Consequently there is at least a diverging sequence of times $\lbrace \tau_k\rbrace$ for which 
\begin{equation}\label{conv:D}
\displaystyle \lim_{k\to \infty} \cD[f(\cdot,\tau_k)] = 0.
\end{equation}
As $\cD[f] = I(f,f^\infty)-\Lambda^2 $ the finiteness of the initial entropy cannot be used directly in order to state that $\Lambda^2$ or $I(f,f^\infty)$ are $L^1(\R_+)$ functions of time.

However, if $f^0\in \cP_2(\R_+),$ thanks to Proposition~\ref{prop:Lambda.D}, we have that $I(f,f^\infty)$ is bounded from above by a constant multiple of $\cD$. As a consequence, it is an $L^1(\R_+)$ function of time and there is at least a diverging sequence of times $\lbrace \tau_k\rbrace$ for which 
\begin{equation}\label{conv:Fisher}
\displaystyle \lim_{k\to \infty} I(f(\cdot,\tau_k),f^\infty) = 0.
\end{equation}
We can then exploit the inequality which relates the Fisher information to the Hellinger distance proved in Proposition ~\ref{prop:Hellinger_Fisher}, in order to prove $L^1$ convergence. 
In general we have that
\[
\begin{aligned}
\|f-g\|_{L^1(\R_+)}
&=\int_0^\infty
|\sqrt f-\sqrt g|(\sqrt f+\sqrt g)\,dv\\
&\le d_H(f,g)
\left(\int_0^\infty(\sqrt f+\sqrt g)^2\,dv\right)^{1/2}\\
&\le 2d_H(f,g).
\end{aligned}
\]
Then, exploiting Propositions~\ref{prop:Hellinger_Fisher}-\ref{prop:Lambda.D} we have that 
\[
\| f(\cdot,t)-f^\infty\|_{L^1(\R_+)} \le 2\,\sqrt{2}\,\mu\,  \sqrt{\left(1 + \frac{C}{\mu^2}\right)\,\cD[f(\cdot,t)]}
\]
so that
\begin{equation}\label{conv:L1}
\displaystyle \lim_{k\to \infty} ||f(\cdot,\tau_k)- f^\infty ||_{L^1(\R_+)} =0.
\end{equation}
Analogously, from Proposition~\ref{prop:Lambda.D}, if $f^0 \in \cP_2(\R_+)$, then $\Lambda^2 \in L^1(\R_+)$ and we can again conclude that there exists at least a diverging sequence of time $\lbrace \tau_k\rbrace$ for which
\begin{equation*}
\displaystyle \lim_{k\to \infty} \Lambda^2(\tau_k) = 0,
\end{equation*}
which implies in turn that
\begin{equation}\label{conv.trace}
\displaystyle \lim_{k\to \infty} f(0,\tau_k) = f^\infty(0).
\end{equation}

The convergence of the $L^1$ distance and the convergence of the boundary value have been proved (without extractions) in~\cite{calvez_analysis_2012}. However, we find it useful to show the convergence (at least up to subsequences) using the proved bounds on the competing terms of the entropy production terms, and using classical arguments of kinetic theory for multi-agent systems~\cite{furioli_2017}. 
Furthermore, the relation between the Hellinger distance and the Fisher information which was proved in Proposition~\ref{prop:Hellinger_Fisher}, together with Proposition~\ref{prop:Lambda.D} allows concluding that
\[
d_H^2(f,f^\infty) \le \dfrac{2}{\mu^2}\left(1+\dfrac{\cM_2}{\mu^2}\right)\cD.
\]
Then we can state convergence of the Hellinger distance, at least up to subsequences
\begin{equation}\label{conv:Hellinger}
\displaystyle \lim_{k\to \infty} \dd_H(f(\cdot,\tau_k),f^\infty) = 0.
\end{equation}

To conclude, we remark that if $f,g$ are solutions to the BDY PDE~\eqref{eq:mainPDE}, then they have the same mass and average. As a consequence, their Fourier distance~\eqref{def:fourier.dist} $d_s$ is well defined for $0 \le s \le 2$. Moreover, if the second moment is uniformly bounded, then the following interpolation can be considered~\cite{gabetta_1995}
\begin{equation}\label{interpolation}
    d_s(f,g) \le d_0(f,g)^{1-s/r}\, d_r(f,g)^{s/2}, \qquad 0<s<r.
\end{equation}
As a consequence, we can state that
\begin{theorem}
    Let $f$ be the solution to the BDY~\eqref{eq:mainPDE} corresponding to an initial condition $f^0 \in \cP_2(\R_+)$. 
    As a consequence we have that
    \[
    \lim_{t \to \infty}d_s(f(\cdot,t),f^\infty) = 0 \qquad \forall s \in [0,2).
    \]
\end{theorem}
\begin{proof}
    First of all, we remark that
    \[
    d_0(f,f^\infty) \le \|f-f^\infty \|_{L^1(\R_+)},
    \]
    which allows one to obtain the thesis for $s=0$. For $s \in (0,2)$,  as $f^0 \in\cP_2(\R_+)$, then $f$ has a uniformly bounded second moment and the interpolation~\eqref{interpolation} can be applied with $r=2$ at most, so that
    \[
    d_s(f,f^\infty) \le d_0(f,f^\infty)^{1-s/2} d_2(f,f^\infty)^{s/2}, \qquad \forall s \in (0,2).
    \]
    The latter tends to 0 thanks to the convergence of $d_0(f,f^\infty)$ and to the finiteness of $d_2(f,f^\infty)$.
\end{proof}

\subsection{Linearization around equilibrium}\label{subsec:sec3.4}

We perform a linearization analysis around the Boltzmann–Gibbs equilibrium distribution $f^\infty$ for solutions of the BDY PDE \eqref{eq:mainPDE}, and establish an explicit rate of convergence within the linearized (weighted $L^2$) framework. To this aim, we assume the following ansatz:
\begin{equation*}
f(v,t) = f^\infty(v) + \varepsilon\,r(v,t) \quad \textrm{with $|\varepsilon| \ll 1$},
\end{equation*}
we have $\partial_v f(v,t) = \partial_v f^\infty(v) + \varepsilon\,\partial_v r(v,t)$, $\partial_{vv} f(v,t) = \partial_{vv} f^\infty(v) + \varepsilon\,\partial_{vv} r(v,t)$, and $f(0,t) = f^\infty(0) + \varepsilon\,r(0,t)$. In addition, Lemma \ref{lemma:g.P2-advection} guarantees that
\[\int_0^\infty r(v,t)\,\dd v = 0 \quad \textrm{and} \quad \int_0^\infty v\,r(v,t)\,\dd v = 0~~\textrm{for all $t\geq 0$}.\]
Taking the limit as $\varepsilon \to 0$ yields the following linearized PDE from the BDY PDE problem \eqref{eq:mainPDE}:
\begin{equation}\label{eq:linearPDE}
\begin{cases}
\partial_t r(v,t) = \partial_{vv} r(v,t) + f^\infty(0)\,\partial_v r(v,t) + r(0,t)\,\partial_v f^\infty(v), &~~ v > 0,~t\geq 0,\\
\partial_v r(v,t) + 2\,f^\infty(0)\,r(v,t) = 0, &~~ v = 0,~t\geq 0.
\end{cases}
\end{equation}
For the linearized equation \eqref{eq:linearPDE}, the natural entropy/energy is the $L^2(f^{-1}_\infty)$ norm of $r(\cdot,t)$ defined by
\begin{equation*}\label{eq:weighted_L2}
\mathcal{E}[r(\cdot,t)] \coloneqq \frac 12\,\int_0^\infty \frac{r^2(v,t)}{f^\infty(v)}\,\dd v.
\end{equation*}
Our main goal in this section is to establish the following quantitative convergence guarantee satisfied by the solution of the linearized problem \eqref{eq:linearPDE}:
\begin{theorem}\label{thm:expo_conv}
Assume that $r(v,t)$ is a classical solution of the linear PDE \eqref{eq:linearPDE}. Then for all $t\geq 0$ it holds that
\begin{equation*}
\mathcal{E}[r(\cdot,t)] \leq \mathcal{E}[r(\cdot,0)]\,\expo^{-\frac{t}{6\,\mu^2}}.
\end{equation*}
\end{theorem}
The remainder of this section is devoted to the proof of Theorem \ref{thm:expo_conv}. To facilitate the presentation, we first establish two preliminary lemmas, which may also be of independent interest.

\begin{lemma}\label{lem:L1}
Assume that $r \colon \mathbb{R}_+ \to \mathbb R$ is continuously differentiable on $[0,\infty)$ and satisfies $\int_0^\infty r(v)\, \dd v = 0$ and $\int_0^\infty v\,r(v)\, \dd v = 0$. Then
\begin{equation}\label{eq:E1}
r^2(0) \leq \frac 13\,\int_0^\infty \frac{|r'(v)|^2}{f^\infty(v)}\,\dd v.
\end{equation}
\end{lemma}

\begin{proof}
For any $\alpha,\beta \in \mathbb R$, the integral constraints on $r$ together with the Cauchy-Schwarz inequality imply that
\begin{equation*}
\begin{aligned}
r^2(0) &= \left(\int_0^\infty \left(1+\alpha\,v+\beta\,v^2\right)\,r'(v)\,\dd v\right)^2 \\
&\leq \int_0^\infty \frac{|r'(v)|^2}{f^\infty(v)}\,\dd v\,\int_0^\infty f^\infty(v)\,\left(1+\alpha\,v+\beta\,v^2\right)^2\,\dd v.
\end{aligned}
\end{equation*}
Let $g(\alpha,\beta) \coloneqq \int_0^\infty f^\infty(v)\,\left(1+\alpha\,v+\beta\,v^2\right)^2\,\dd v$. Explicit evaluations of the moments of the Boltzmann-Gibbs distribution $f^\infty$ \eqref{eq:BG} yield that
\begin{equation*}
\begin{aligned}
g(\alpha,\beta) &= \int_0^\infty \left(1+2\,\alpha\,v+(\alpha^2+2\,\beta)\,v^2+2\,\alpha\,\beta\,v^3+\beta^2\,v^4\right)\,f^\infty(v)\,\dd v \\
&= 1 + 2\,\alpha\,\mu + 2\,(\alpha^2+2\,\beta)\,\mu^{2} + 12\,\alpha\,\beta\,\mu^{3} + 24\,\beta^2\,\mu^{4}.
\end{aligned}
\end{equation*}
Elementary calculus shows that the function $g$ admits a unique critical point at $(\alpha,\beta) = \left(-\mu^{-1},(6\,\mu^2)^{-1}\right)$, and moreover:
\[\min_{(\alpha,\beta)\in \mathbb{R}^2} g(\alpha,\beta) = g\left(-\frac{1}{\mu},\frac{1}{6\,\mu^2}\right) = \frac 13.\]
This leads to the advertised bound \eqref{eq:E1}.
\end{proof}

\begin{lemma}[Poincar\'e-type inequality]\label{lem:L2}
Assume that $r \colon \mathbb{R}_+ \to \mathbb R$ is continuously differentiable on $[0,\infty)$. Then
\begin{equation}\label{eq:E2}
\int_0^\infty \frac{|r(v)|^2}{f^\infty(v)}\,\dd v \leq 4\,\mu^2\,\int_0^\infty \frac{|r'(v)|^2}{f^\infty(v)}\,\dd v.
\end{equation}
\end{lemma}

\begin{proof}
We notice that
\begin{align*}
&\int_0^\infty \frac{|r(v)|^2}{f^\infty(v)}\,\dd v = \int_0^\infty \frac{1}{f^\infty(v)}\left(\int_v^\infty r'(y)\,\dd y\right)^2\,\dd v \\
&= \int_0^\infty \frac{1}{f^\infty(v)}\,\int_{z\geq v} r'(z)\,\dd z\,\int_{w\geq v} r'(w)\,\dd w\,\dd v \\
&= \int_{z\geq 0} \int_{w\geq 0} r'(z)\,r'(w)\,\int_0^{\min(z,w)} \frac{1}{f^\infty(v)} \,\dd v\,\dd z\,\dd w \\
&= \int_{z\geq 0} \int_{w\geq 0} r'(z)\,r'(w)\,\mu^2\,\left(\expo^{\frac{\min(z,w)}{\mu}}-1\right)\dd z\,\dd w \\
&= 2\,\mu^2\,\int_{z\geq 0} r'(z)\,\left(\expo^{\frac{z}{\mu}}-1\right)\,\underbrace{\int_{w\geq z} r'(w)\,\dd w}_{=~-r(z)}\,\dd z \\
&\leq 2\,\mu^2\,\int_{z\geq 0} |r'(z)|\,|r(z)|\,\expo^{\frac{z}{\mu}}\,\dd z = 2\,\mu\,\int_{z\geq 0} \frac{|r'(z)|}{\sqrt{f^\infty(z)}}\,\frac{|r(z)|}{\sqrt{f^\infty(z)}}\,\dd z \\
&\leq 2\,\mu\,\left(\int_0^\infty \frac{|r'(z)|^2}{f^\infty(z)}\,\dd z\right)^{\frac 12}\,\left(\int_0^\infty \frac{|r(z)|^2}{f^\infty(z)}\,\dd z\right)^{\frac 12},
\end{align*}
from which the claimed bound \eqref{eq:E2} follows immediately.
\end{proof}

\noindent We now have all the necessary ingredients to present the proof of Theorem \ref{thm:expo_conv}. \\

\noindent{\bf Proof of Theorem \ref{thm:expo_conv}~:}~ The evolution of $\mathcal{E}[r(\cdot,t)]$ is dictated by
\begin{equation}\label{eq:start}
\begin{aligned}
\frac{\dd}{\dd t} \mathcal{E}[r(\cdot,t)] &= \int_0^\infty \frac{\partial_{vv} r(v,t) + f^\infty(0)\,\partial_v r(v,t) + r(0,t)\,\partial_v f^\infty(v)}{f^\infty(v)}\,r(v,t)\,\dd v \\
&= \int_0^\infty \frac{\partial_{vv} r(v,t) + f^\infty(0)\,\partial_v r(v,t)}{f^\infty(v)}\,r(v,t)\,\dd v,
\end{aligned}
\end{equation}
where the last identity follows from the relation $\dfrac{\partial_v f^\infty(v)}{f^\infty(v)} = -\mu^{-1}$ for all $v \geq 0$,
together with the fact that $\int_0^\infty r(v,t) \,\dd v = 0$ for all $t \geq 0$. Next, we observe that
\begin{equation}\label{eq:middle}
\begin{aligned}
&\int_0^\infty \frac{\partial_{vv} r(v,t)}{f^\infty(v)}\,r(v,t)\,\dd v = -\frac{r(0,t)}{f^\infty(0)}\,\partial_v r(0,t) - \int_0^\infty \partial_v r(v,t)\,\partial_v \left(\frac{r(v,t)}{f^\infty(v)}\right)\,\dd v \\
&= -\frac{r(0,t)}{f^\infty(0)}\,\partial_v r(0,t) - \int_0^\infty \frac{|\partial_v r(v,t)|^2}{f^\infty(v)}\,\dd v - \int_0^\infty \frac{f^\infty(0)\,\partial_v r(v,t)}{f^\infty(v)}\,r(v,t)\,\dd v.
\end{aligned}
\end{equation}
Inserting \eqref{eq:middle} into \eqref{eq:start} and invoking the boundary condition \eqref{eq:linearPDE}, we obtain
\begin{equation*}\label{eq:evolution_of_energy}
\frac{\dd}{\dd t} \mathcal{E}[r(\cdot,t)] = -\frac{r(0,t)}{f^\infty(0)}\,\partial_v r(0,t) - \int_0^\infty \frac{|\partial_v r(v,t)|^2}{f^\infty(v)}\,\dd v = 2\,r^2(0,t) - \int_0^\infty \frac{|\partial_v r(v,t)|^2}{f^\infty(v)}\,\dd v.
\end{equation*}
Applying Lemma \ref{lem:L1} and Lemma \ref{lem:L2} in succession, we end up with the differential inequality $\frac{\dd}{\dd t} \mathcal{E}[r(\cdot,t)] \leq -\frac{1}{6\,\mu^2}\,\mathcal{E}[r(\cdot,t)]$, thus the proof of Theorem \ref{thm:expo_conv} is completed by a standard application of Gronwall's inequality. \qed

\section{Conclusion}\label{sec:sec4}

In this paper, we derived a continuous version of the Bennati–Dragulescu–Yakovenko (BDY) money exchange model, originally formulated on the discrete state space of non-negative integers. Although the BDY model is one of the earliest and most influential frameworks in the econophysics literature and forms the basis for a wide variety of subsequent generalizations, there has (to the best of our knowledge) been no systematic derivation of its counterpart in a continuous wealth space, where admissible wealth values range over $\mathbb{R}_+$. Developing such a formulation is essential for connecting the classical BDY dynamics with tools from nonlinear PDEs, kinetic theory, and continuous mean-field descriptions.

We start from a collision-like kinetic model (the $\varepsilon
$-BDY PDE) which admits a unique solution, but which does not have a unique equilibrium. Somewhat unexpectedly, the quasi-invariant limit procedure leads us to the formulation of a nonlinear Fokker–Planck type PDE~\eqref{eq:mainPDE} on $\R_+$, 
supplemented with a nonlinear Robin-type boundary condition which ensures the conservation of total mass and average wealth. The equation features a constant diffusion coefficient and a nonlinear drift term which is the boundary value reflecting the underlying microscopic exchange mechanism. We further demonstrated that this PDE inherits several qualitative properties of the original BDY process at the mean-field level. In particular and to some extent, its evolution parallels the mean-field ODE system associated with the discrete model and preserves several key structural features of the exchange dynamics.

Finally, we proved that the solution of the PDE problem~\eqref{eq:mainPDE} converges (along a subsequence of diverging times) to its unique equilibrium distribution characterized by the classical Boltzmann–Gibbs (exponential) law, which 
is an equilibrium for both the collision-like and the Fokker-Planck equation. This has been proved by means of entropy arguments, which posed some new challenges. 

This work establishes a rigorous bridge between the discrete stochastic BDY model which takes into account exchange of absolute quantities and a continuous deterministic description, providing a unified PDE framework for analyzing more complex extensions of wealth exchange dynamics, such as the one with probabilistic cheaters, or the rich/poor-biased ones.

Our work also leaves a number of compelling open problems for future investigation. A first natural question is whether one can prove that the Hellinger distance itself is monotone along solutions, although full-time convergence in Hellinger distance follows readily from the known $L^1$ convergence. In fact, establishing the monotonicity of the Hellinger distance is far from trivial: the presence of the time-dependent term $\Lambda$, whose sign may change over time, prevents a direct argument. The same question arises for the convergence and monotonicity of the $d_2$ Fourier distance. \\

\noindent {\bf Acknowledgement.} Fei Cao gratefully acknowledges support from an AMS-Simons Travel Grant, administered by the American Mathematical Society with funding from the Simons Foundation. Nadia Loy is a member of GNFM-INdAM. \\

\noindent {\bf Disclosure Statement.} The authors declare that they have no known competing financial interests or personal relationships that could have appeared to influence the work reported in this paper.




\end{document}